\documentclass[reqno]{amsart}
\usepackage{xcolor}
\usepackage{amsthm,amsmath,amssymb,latexsym,soul,cite,mathrsfs}
\usepackage{enumitem}
\usepackage{color,enumitem,graphicx}
\usepackage[colorlinks=true,urlcolor=blue,
citecolor=red,linkcolor=blue,linktocpage,pdfpagelabels,
bookmarksnumbered,bookmarksopen]{hyperref}
\usepackage[english]{babel}

\usepackage[left=2.7cm,right=2.7cm,top=2.9cm,bottom=2.9cm]{geometry}

\usepackage[hyperpageref]{backref}

\numberwithin{equation}{section}

\newtheorem{theorem}{Theorem}[section]
\theoremstyle{plain}
\newtheorem{theoremletter}{Theorem}

\newtheorem{lemma}[theorem]{Lemma}
\newtheorem{lemmaletter}[theoremletter]{Lemma}
\newtheorem{corollary}[theorem]{Corollary}
\newtheorem{proposition}[theorem]{Proposition}

\newtheorem{remark}[theorem]{Remark}
\newtheorem{definition}[theorem]{Definition}
\newtheorem{example}[theorem]{Example}

\newcommand{\dx}{\,\mathrm{d}x}
\newcommand{\dy}{\,\mathrm{d}y}

\newcommand{\dtau}{\,\mathrm{d}\tau}

\newcommand{\dive}{\mathrm{div}}
\newcommand{\rad}{\mathrm{rad}}
\DeclareMathOperator{\supp}{supp}

\newcommand{\loca}{\operatorname{loc}}

\def\Xint#1{\mathchoice
	{\XXint\displaystyle\textstyle{#1}}%
	{\XXint\textstyle\scriptstyle{#1}}%
	{\XXint\scriptstyle\scriptscriptstyle{#1}}%
	{\XXint\scriptscriptstyle\scriptscriptstyle{#1}}%
	\!\int}
\def\XXint#1#2#3{{\setbox0=\hbox{$#1{#2#3}{\int}$}
		\vcenter{\hbox{$#2#3$}}\kern-.5\wd0}}

\def\dashint{\Xint-}

\usepackage[norefs,nocites]{refcheck}

\title[Systems of Schr\"{o}dinger equations]{Concentration-compactness via profile decomposition for systems of coupled Schr\"{o}dinger equations of Hamiltonian type}

\author[A. Cardoso]{Anderson Cardoso}
\address{Department of Mathematics,\\
	Federal University of Sergipe 49100-000, S\~ao Crist\'ov\~ao-SE, Brazil}
\email{anderson@mat.ufs.br} 

\author[J.M. do \'O]{Jo\~ao Marcos do \'O}
\address{Department of Mathematics,\\
	Federal University of Para\'{\i}ba 58051-900,  Jo\~ao Pessoa-PB, Brazil}
\email{jmbo@ufpb.br}

\author[D.~Ferraz]{Diego Ferraz}
\address{Department of Mathematics,
	Federal University of Rio Grande do Norte,
	59078-970, Natal-RN, Brazil}
\email{diego.ferraz.br@gmail.com}

\thanks{Corresponding author: Jo\~ao Marcos do \'O}

\thanks{The second author is supported by National Council for Scientific and Technological Development (CNPq) \#312340/2021-4 and \#429285/2016-7 and  Para\'iba Research Foundation (FAPESQ) \#3034/2021}

\date{\today}

\begin{document}
\begin{abstract}
We analyse Hamiltonian-type systems of second-order elliptic PDE invariant under a non-compact group and, consequently, involve a lack of compactness of the  Sobolev embedding. 
We show that the loss of compactness can be compensated by using a concentration-compactness principle via weak profile decomposition for bounded Palais-Smale sequences in Banach spaces. Our analysis to prove the existence of ground states involves a reduction by the inversion method of the system to a fourth-order equation combined with a variational principle of a minimax nature. Among other results, including regularity and a Pohozaev-type identity, we also prove the non-existence of weak solutions for a class of Lane-Emden systems. 
\end{abstract}

\maketitle

	\section{Introduction}
	
	%
	%

	The purpose of this paper is to study the following class of Hamiltonian systems of elliptic partial differential equations 
	\begin{equation}\label{SH}\tag{$S_H$}
		\left\{
		\begin{aligned}
			-\Delta v+ V(x)v &= H_{u}(x,u,v) \quad &\mbox{in}\quad \mathbb{R}^N,\\
			-\Delta u+ V(x)u &= H_{v}(x,u,v) \quad &\mbox{in}\quad \mathbb{R}^N,
		\end{aligned}
		\right.
	\end{equation}
	where $N \geq 3$  and $ H:\mathbb{R}^N\times\mathbb{R} \to \mathbb{R}$ is a $C^1$ function. Using variational methods based on the reduction by inversion framework, we obtain the existence of ground state solutions of Syst.~\eqref{SH} involving nonlinearities with subcritical or critical growth in terms of the critical hyperbola. To study the subcritical systems,  we assume that the potential  $ \; V: \mathbb{R}^N \rightarrow \mathbb{R}$ is nonnegative and belongs to the reverse H\"{o}lder class. 
 For the critical case, we consider the inverse square Hardy potential.	
	
	Nonlinear elliptic problems of this form arise in many physical contexts. For more details on the physical background, we refer the reader to \cite{beresI,Cazenave}.
	Such problems are motivated, for instance, when studying the propagation of light in some optical materials. In particular, by the search for certain kinds of standing waves for a system of time-dependent nonlinear Schr\"{o}dinger equations of the form
	\begin{equation}\label{NSS}
		\left\{
		\begin{aligned}
			i\hbar\frac{\partial\psi}{\partial t}=-\frac{\hbar ^2 }{2m} \Delta_x \psi +W(x) \psi- H_{\varphi }(x,\psi, \varphi),\\
			i\hbar\frac{\partial\varphi}{\partial t}=-\frac{\hbar ^2 }{2m} \Delta_x \varphi +W(x) \varphi- H_{\psi }(x,\psi, \varphi), 
		\end{aligned}
		\right.
	\end{equation}
	where  $\hbar$  and $m$ are positive constants, $\varphi,\ \psi : \mathbb{R}^+ \times\mathbb{R}^N \rightarrow \mathbb{C}, \, N \geq 3, W(x)$ is a nonnegative  potential, and the  nonlinear term $H:\mathbb{R}^N \times \mathbb{R}^2 \rightarrow \mathbb{R}$ is $C^1$. Here we are interested in soliton (standing wave) solutions of \eqref{NSS}, namely  solutions with the form 
	$
	\psi(x,t)=  u(x) e^{-itE/\hbar}$   and $  \varphi(x,t)= v(x)e^{-itE/\hbar}
	$
	with  $E\in \mathbb{R}$ and $u, \, v$ being real valued functions.  
        We obtain Syst.~\eqref{SH},  substituting $(\psi(x,t), \varphi(x,t))$ in \eqref{NSS} and assuming  $H_{u_j}(x,e^{i\theta }u_1 , e^{i \theta } u_2 ) = e^{i \theta } H_{u_j} (u_1, u_2)$ and $W(x) = V(x)+E$.
	For some recent research on nonlinear elliptic systems,  see, for instance,  \cite{anderson,edersonGSH,dooH,guide,ruf,hulshof,HMV,felmer,krist} and the references given there. 
	We can see  that   Syst.~\eqref{SH}  can be reduced to a single Schrödinger equation of the form 
	$
	-\Delta u + V(x)u = f(x,u)$ in $  \mathbb{R}^N,
	$
	when $u=v$ and $H_{u}(x,u,v)=H_{v}(x,u,v)=f(x,u)$,
	which originates from the search of certain kinds of stationary states in nonlinear time-dependent equations of Schrödinger type, see \cite{beresI,paper1,paper2,lionscompcase1,lionscompcase2,lionslimitcase1,lionslimitcase2,strauss,tintabook,zouh}.
		

	In our analysis of Syst.~\eqref{SH},  we are interested in the existence of solutions when the nonlinearities have maximal growth in terms of the so-called critical hyperbola, which is the natural counterpart of the Sobolev exponent  $2^\ast-1=(N+2)/(N-2) \, (N \geq 3)$.  Hamiltonian systems of elliptic partial differential equations is a natural extension of the celebrated Lane-Emden equation $-\Delta u = u^p  \mbox{ in } \Omega \subset \mathbb{R}^N (N \geq 3)$ with $u=0$ on $\partial \Omega$, which has  Sobolev exponent as the dividing line for existence and non-existence of positive solutions if $\Omega$ is star-shaped.  It was proved in  \cite{clement,Mitidieri-1993,vandervorst_prim}  that the solution set of Hamiltonian elliptic systems is strongly affected by the associated critical hyperbola, which gives the natural assumption on the growth of nonlinearities.  Precisely, in \cite{Mitidieri-1993,vandervorst_prim}, by using a generalization of the Pohozaev identity, it was demonstrated that  positive classical solutions for the following two coupled semilinear Poisson equations
	\begin{equation}\label{So}\tag{$S_\Omega$}
		\left\{
		\begin{aligned} 
			-\Delta v &= f(u) \quad &\mbox{in}&\quad \Omega,\\
			-\Delta u &= g(v) \quad &\mbox{in}&\quad \Omega,\\
			u=v&=0,        \quad &\mbox{on}&\quad \partial\Omega,
		\end{aligned}
		\right.
	\end{equation}
	with $f(u)=|u|^{p-2}u, \; g(v)=|v|^{q-2}v$  on star-shaped domains can only exist if $(p,q)$ satisfies 
	\begin{equation}\label{HYPsub}\tag{$\mathcal{H}_{sub}$}
		\frac{1}{p}+\frac{1}{q}> 1-\frac{2}{N},\quad p,q>1,
	\end{equation}
	while the true notion of criticality is given for $(p,q)$ such that
	\begin{equation}\label{HYPcrit}\tag{$\mathcal{H}_{crit}$}
		\frac{1}{p}+\frac{1}{q}= 1-\frac{2}{N},\quad p,q>1,\ N \geq 3,
	\end{equation}
	i.e., when $(p,q)$ lies on the so called Sobolev \emph{critical hyperbola}.   
	
	
For the subcritical case in the sense of condition \eqref{HYPsub}, Ph. Cl\'{e}ment et al. \cite{clement} by using topological arguments proved the existence of positive solutions of \eqref{So} in the superlinear regime. On the other hand, also for the subcritical case, using variational arguments based on a linking theorem was proved in \cite{felmer,hulshof} existence results for \eqref{So}. In \cite{HMV}, J.~Hulshof et al. study \eqref{So} for $f(u)=\lambda u +|u|^{p-2}u $ and $g(v)= \mu v + |v|^{q-2}v ,$  with critical growth rates  given by  \eqref{HYPcrit}, where $\lambda, \mu$ are real parameters. Based on an adaptation of the dual variational method  under suitable conditions on the linear terms, the authors obtained a natural generalization of the famous Br{\'e}zis-Nirenberg problem  \cite{Brezis-Nirenberg}.
		

	There are some challenges in studying  Hamiltonian systems of elliptic partial differential equations. In \cite{felmer,hulshof,ruf}, the authors explained the difficulty of implementing the same minimax methods to study the scalar equations. Despite this, as already emphasized in these works, problems like  \eqref{So} and \eqref{SH} allow for a variational formulation, and solutions arise as critical points of the associated Euler-Lagrange functionals, which have quadratic parts strongly indefinite.  Moreover, we must handle the lack of compactness when the problems are defined in the whole Euclidean space, or the nonlinearities have critical growth.  We can find some ways to attack problems like \eqref{So} using variational methods.  Each approach has different advantages and disadvantages, as was explained in \cite{guide}. Here, we choose a suitable method to study the existence of ground states and derive specific qualitative properties for more general nonlinearities.
	

	We develop our analysis via reduction by inversion, which was studied previously for bounded domains in \cite{krist,edersonGSH,dooH,guide}. Heuristically, this method consists in transform a system like \eqref{SH}, or \eqref{So}, in a  fourth--order equation by isolating $u$,  taking the explicit value of $v.$ For example, in Syst.~\eqref{So} we isolate 
	\begin{equation*}
		\fbox{
			$
			v = |(-\Delta u)|^{m-2} (-\Delta u)\quad \text{with} \quad m:=q/(q-1),
			$
		}
	\end{equation*}
	and replace it on the first equation of \eqref{So} to obtain a  fourth--order equation of the form
	\begin{equation}\label{P_o}
		-\Delta(|(-\Delta u)|^{m-2} (-\Delta u)) = |u|^{p-2}u,\ u=\Delta u = 0 \mbox{ on } \partial\Omega.
	\end{equation}
	Solutions of \eqref{P_o} are now solutions for Syst.~\eqref{So}. The great advantage of this method is that one can deal with the sublinear case ($1<p,\, q\leq 2$) and the general case of critical/subcritical hyperbola without asking superlinear conditions over the nonlinearities. Therefore, one can see that our results ($q=2$) also apply to elliptic problems involving the biharmonic operator of the form $\Delta ^2 u = f (x,u)$ in $\mathbb{R}^N.$ On the other hand, using this method \eqref{P_o} becomes more intricate since the associated energy functional has to be defined on a second-order Sobolev space and has a quasilinear term.
	

	In this work we focus our analysis in Syst.~\eqref{SH} when $H_u(x,u,v)=f(x,u)$ and $H_v(x,u,v)=b(x)g(v)$, namely
	\begin{equation}\label{S}\tag{$S$}
		\left\{
		\begin{aligned}
			-\Delta v+ V(x)v &= f(x,u) \quad &\mbox{in}\quad \mathbb{R}^N,\\
			-\Delta u+ V(x)u &= b(x)g(v) \quad &\mbox{in}\quad \mathbb{R}^N.
		\end{aligned}
		\right.
	\end{equation}

    To use the inversion method to transform Syst.~\eqref{S} to a single fourth-order equation, it is natural to assume that $b(x)\in L^\infty (\mathbb{R}^N),$ $0<\beta \leq b(x)$ almost everywhere in $\mathbb{R}^N,$ and $g(t)$ is a nondecreasing homeomorphism. In this way, setting  
	\begin{equation*}
		\fbox{
			$
			L(u) := -\Delta u +V(x) u \quad \text{ and } \quad  h(t) := g^{-1} (t),
			$
		}
	\end{equation*}
	we can reduce the search of ground state solutions for \eqref{S} to prove the existence of solutions for  the fourth-order equation
	\begin{equation}\label{P}
		L \left[h\left(\frac{1}{b(x)} L(u) \right) \right] = f (x,u)\quad \text{ in } \quad \mathbb{R}^N.
		\tag{$P$}
	\end{equation}


	As we already mention, in this paper, we deal basically with two classes of nonlinearities.  Namely, we consider \eqref{S} for the case that $f(x,t)$ and $g(t)$ have subcritical growth in the sense of \eqref{HYPsub}, and also with critical growth as defined in \eqref{HYPcrit}.  Moreover, we consider the case when the nonlinearities have oscillating behavior about the pure powers $|t|^{p-2}t$ and $|t|^{q-2}t$,  see Sections \ref{s_subcritical} and \ref{s_critical} for precise assumptions.


	For the subcritical case, we improve and complement some results in \cite{edersonGSH}, where it studied the case that $V(x)\equiv b(x) \equiv 1,$ $f(x,t)\equiv |t|^{p-2}t$ and $g(t) \equiv |t|^{q-2}t.$ We also prove some local regularity results for Syst.~\eqref{S}. More precisely, based on some recent results about the regularity of very weak solutions proved in  \cite{veryweak}, we can reduce the local regularity analysis of \cite{edersonGSH} to a  bootstrap argument involving the solutions of Syst.~\eqref{S} (see Theorem \ref{p_regsub}), in a similar fashion of the scalar equations. For this case, we assume that our potential $V(x)$ belongs to the reverse H\"{o}lder class, which is an appropriate class of potentials to apply the reduction inversion method. It is also worth calling attention to the bounded domain case of \eqref{S} studied in \cite{dooH} ($V(x)\equiv 0$), where it was considered a general approach using the reduction by inversion. Our results complement most of the ones found in the literature since $p$ and $q$ cover most of the regions delimited in \eqref{HYPsub}, see Theorem \ref{GS_subcrit}, and  \cite{ruf,ramos}.


	Nonlinear elliptic problems involving critical growth have been studied in recent years, motivated by several applications, and the mathematical challenges that arise when addressing this class of problems via variational methods, see \cite{Brezis-Nirenberg,willemcomp,guide,edersonGSH}.  
	Few papers deal with Hamiltonian elliptic systems involving critical growth. Notably, most of the results of the literature on this class of systems when defined in the whole Euclidean space $\mathbb{R}^N$ involve both nonlinearities superlinear perturbations of the pure power $|t|^{2^*-2}t,$ where $2^\ast = 2N/(N-2).$
	For some results on the case of pure power nonlinearities on the critical hyperbola, we cite \cite{anderson,vanderCRIT,serrinzou}. 
	As a consequence of \cite[Theorem I.1]{lionscompcase1}, we can see that \eqref{So} for $\Omega=\mathbb{R}^N,  \; f(u)=|u|^{p-2}u$ and $g(v)= |v|^{q-2}v $ with $(p,q)$ satisfying \eqref{HYPcrit} has a ground state which is unique up to scalings and translations. 
	Moreover, the ground state is positive, radially symmetric and decreasing in the radial component.
	In \cite{vanderCRIT,lionslimitcase1}, studying a constrained minimization problem in combination with the well-known P.-L. Lions concentration-compactness principle \cite{lionscompcase1,lionscompcase2,lionslimitcase1,lionslimitcase2}, the authors studied \eqref{S} when $V(x)\equiv 0,$ $b(x)\equiv 1,$ $f(x,t)\equiv |t|^{p-2}t$ and $g(t) = |t|^{q-2}t.$ 
	On the other hand, in \cite{serrinzou}, the authors studied \eqref{S} considering autonomous nonlinearities $f(t)$ and $g(t)$ under general assumptions, and also taking $V(x)\equiv 0.$ In that paper, the authors obtained the existence of radial solutions by using topological methods, and so $C^2-$regularity on the nonlinearities was required. For results on Hamiltonian systems that apply variational methods and deal with nonlinearities with critical growth and nonzero potential, we only find \cite{anderson}, where the authors used a Lyapunov-Schmidt type reduction supposing that $b(x)\equiv 1$ and $f(x,t) \equiv g(t)\equiv |t|^{2 ^\ast -2} t,$ also assuming mild assumptions on $V(x),$ allowing this potential to null in nonempty interior sets of $\mathbb{R}^N.$


 As already mentioned, to study the subcritical systems,  we assume that the potential  $ \; V: \mathbb{R}^N \rightarrow \mathbb{R}$ is nonnegative and belongs to the reverse H\"{o}lder class.   For the critical case, we consider the inverse square Hardy potential, more precisely, $V(x) \equiv V_\lambda(x) \equiv  -\lambda |x|^{-2},$ where $\lambda \in [0,\Lambda _{N,m}^{1/m} )$ and $\Lambda _{N,m}$ is the sharp constant of the Hardy inequality
	\begin{equation*}
		\Lambda _{N,m} \int _{\mathbb{R}^N}|x|^{-2m}|u|^{m} \dx \leq \int _{\mathbb{R}^N}|\Delta u|^m \dx,\quad\forall \, u \in C^\infty _0 (\mathbb{R}^N),\ m>1.
	\end{equation*}
	See \cite[Theorem 2.4]{lionslimitcase2} for more details. For all $\lambda \in [0,\Lambda _{N,m}^{1/m} ),$ we also provide a new argument and prove, under suitable conditions, that solutions of \eqref{S} have $C^{2,\alpha}_{\loca } (\mathbb{R}^N\setminus\{ 0 \})$ regularity (see Theorem \ref{c_regularity_crit_0}), and when $\lambda = 0$ they belong to $C^{2,\alpha}_{\loca } (\mathbb{R}^N),$ improving and complementing the regularity results of \cite{hulshof,vandervorst_best}.


	To overcome the lack of compactness in our problems, we approach a concentration-compactness principle through profile decomposition (Theorems \ref{teo_profcrit} and \ref{teo_profsub}). 
 Roughly speaking, this result refines the methods given in \cite{lionscompcase1,lionscompcase2,lionslimitcase1,lionslimitcase2}. It describes where bounded sequences of the related Sobolev space may lose compactness concerning the continuous embedding where that space lies, i.e., the analytic description of the weak convergence of sequences in terms of a collection of functions (so-called profiles) that must satisfy suitable conditions. See \cite{tintabanach,wavelet,paper1} for references on profile decompositions type results. Thanks to this new approach, we could consider general oscillating nonlinearities in the critical case (see Section \ref{s_selfsimilar}), which includes the power nonlinearities as a particular case. The core idea is that this class of nonlinearities and potential $V_\lambda (x)$ allow the energy functional of \eqref{P} to be invariant to suitable discrete dilations. So our concentration-compactness type arguments can be applied. However, one must analyse the quasilinear term of the energy functional associated with \eqref{P}. Inspired by some well-known arguments and exploring the additional compactness given by the radial subspace, we prove almost everywhere convergence for this quasilinear term, up to subsequences for bounded Palais-Smale sequences (see Proposition \ref{p_aeconv}). The convexity-type conditions usually imposed to use the Nehari constraint are not required here. Moreover, the increasing monotonicity of the function $t \mapsto f(x,t)t^{-1}$ is not vital in our argument.


Our method to prove the existence of ground states requires a new Pohozaev-type identity to control the energy levels of the associated functional  (see Theorem \ref{p_pohozaev_id}).   Moreover, this identity enables us to study nonexistence results for the supercritical case of Syst.~\eqref{S}. We prove a general nonexistence result that provides, as a particular case, a partial answer (Sobolev sense) for the Lane-Emden conjecture, which involves the region above the critical hyperbola,
	\begin{equation}\label{HYPsuper}\tag{$\mathcal{H}_{sup}$}
		\frac{1}{p}+\frac{1}{q}< 1-\frac{2}{N},\quad p,q>1,\ N \geq 3.
	\end{equation}
See Proposition \ref{p_nonexist}, \, Corollary \ref{c_leconject} and  \cite{zouh,souplet} for the Lane-Emden conjecture.

Another novelty of our regularity and concentration-compactness results is that they can aid in the approach for elliptic problems involving quasilinear terms of second-order Sobolev spaces. We believe that with these methods, it is possible to complement and improve some results of \cite{barile2016,barile2017} concerning the study of solutions for the general Lane-Emden elliptic system of the form
	\begin{equation*}
		\left\{
		\begin{aligned}
			-\Delta v &= - \varrho (x) |u|^{m-2}u + f(x,u) \quad &\mbox{in}\quad \mathbb{R}^N,\\
			-\Delta u &= |v|^{q-2}v \quad &\mbox{in}\quad \mathbb{R}^N,\\
			\quad &u,v \rightarrow 0\text{ as }|x|\rightarrow + \infty,
		\end{aligned}
		\right.
	\end{equation*}
where $N \geq 3,$ $p>1,$ $m=q/(q-1),$ $\varrho (x)$ is a suitable weight and $f(x,t)$ is a nonlinearity that can have critical growth in the hyperbola sense.

\subsection{Terminology and preliminaries definitions}
	%
	%
	
	In this section, we describe the variational formulation of problems \eqref{S} and \eqref{P}.  
	For $\theta >1,$ let us consider the homogeneous Sobolev space $\mathcal{D}^{2,\theta} (\mathbb{R}^N)$ defined as the completion of $C^\infty_0(\mathbb{R}^N)$ with respect to the norm $\| \Delta u \| _\theta$, where $\| \cdot \| _\theta $ denotes the usual Lebesgue norm.
	Associated to the potential $V(x)$,  define $W^{2,\theta}_V(\mathbb{R}^N)$ as the completion of $C^\infty_0(\mathbb{R}^N)$ with respect to the norm $\|u\|_V :=\left( \| \Delta u\|^\theta_\theta + \| V(x) u \|_\theta ^\theta\right)^{1/\theta}.$ 
	Under suitable conditions (see Proposition \ref{p_normasirakov}), we can show that $W^{2,\theta}_V(\mathbb{R}^N)$ is continuously embedded into $W^{2,\theta}(\mathbb{R}^N)$.

	%
	%
	
	The Euler-Lagrange functional associated with \eqref{S} is given by
	\begin{equation*}
		J(u,v) = \int_{\mathbb{R}^N}  \nabla u \cdot \nabla v+V(x) u v \dx - \int_{\mathbb{R}^N} F(x,u)\dx - \int_{\mathbb{R}^N}b(x) G(v) \dx,
	\end{equation*}
	where $F(x,t) = \int _0 ^t f(x, \tau )\dtau, \; G(t) = \int _0 ^t g(\tau )\dtau.$  $J$ is well defined in the spaces
	\begin{center}
		\fbox{
			$
			\mathcal{F}:=W^{2,m}(\mathbb{R}^N) \times  W^{2,l}(\mathbb{R}^N) \quad \text{and} \quad  \mathcal{F}: = \mathcal{D}^{2,m}(\mathbb{R}^N) \times \mathcal{D}^{2,l}(\mathbb{R}^N)\text{  with }  l:=p/(p-1).
			$
		}
	\end{center}
	Additionally, under the hypotheses of our main results, $J$ belongs to class $C^1$ with 
	\begin{multline*}
		J'(u,v) \cdot (\varphi, \psi) = \int_{\mathbb{R}^N}  \nabla u \cdot \nabla \psi+V(x) u \psi \dx + \int_{\mathbb{R}^N}  \nabla v \cdot \nabla \varphi+V(x) v \varphi \dx \\ - \int_{\mathbb{R}^N} f(x,u) \varphi  \dx - \int_{\mathbb{R}^N}b(x)g(v)\psi \dx.
	\end{multline*}
	Hence the critical points of $J$ are precisely the weak solutions of Syst.~\eqref{S}.
	
	%
	%
	
	\begin{definition}A nontrivial weak solution $(u,v) \in \mathcal{F}$ of \eqref{S}  is called a ground state if 
		$$
		J(u,v) = \inf\{ J(\varphi, \psi)  :  (\varphi, \psi) \in  \mathcal{F} \setminus\{(0,0)\}, \;  J'(\varphi, \psi) = 0\}.
		$$
		A strong solution for Syst.~\eqref{S} is a pair $(u,v) \in  L^1_{\loca}(\mathbb{R}^N) \times L^1_{\loca}(\mathbb{R}^N) $ where
		each coordinate is twice weakly differentiable, satisfying \eqref{S} almost everywhere.
	\end{definition}

	%
	%

	The natural functional associated with \eqref{P} is given by
	\begin{equation*}
		I(u) = \int _{\mathbb{R}^N} b(x)H\left( (b(x))^{-1} L(u)\right)  \dx - \int_{\mathbb{R}^N} F(x,u) \dx,
	\end{equation*}
	where $ H(t) = \int _0 ^t h(\tau )\dtau$ and $h(t) = g^{-1} (t)$. The functional $I$ is well defined and belongs to $C^1(E,\mathbb{R})$ for $E = W_V^{2,m}(\mathbb{R}^N)$ and $E=\mathcal{D}^{2,m}(\mathbb{R}^N).$ Moreover, for all $\varphi \in E$,
	\[
	I'(u)\cdot \varphi=\int _{\mathbb{R}^N } h \left((b(x))^{-1}  L(u) \right) L(\varphi)\dx  - \int_{\mathbb{R}^N} f(x,u) \varphi \dx.
	\]
	Thus, critical points of $I$ correspond to weak solutions of \eqref{P}.
	\begin{definition}
		A nontrivial weak solution $u \in E$ is said to be a ground state of \eqref{P} if $$	I(u) = \inf\{ I(\varphi)  :  \varphi \in  E \setminus\{0\}, \;  I'(\varphi) = 0\}.$$
	\end{definition}
	For dealing with the critical case, we consider functional $I$ defined in the space $E=\mathcal{D}^{2,m}(\mathbb{R}^N)$ and restrict $I$ to space $\mathcal{D}_\rad^{2,m}(\mathbb{R}^N).$ We apply the Palais symmetric criticality principle to show the existence of weak solutions of  \eqref{P}.

	%
	%
	
	We make the following hypotheses, where $\mathcal{C}_1,\mathcal{C}_2, \mathcal{C}_3$ are positive constants:
	\begin{flalign}\tag{$g_1$}\label{g_dois}
		\mathcal{C}_1 |t|^{m-2}t \leq g^{-1}(t) \leq \mathcal{C}_2 |t|^{m-2}t,\ \forall \, t \in \mathbb{R}.
	\end{flalign}
	\begin{flalign}\tag{$g_2$}\label{g_tres}
		(g^{-1}(t) - g^{-1}(s))(t-s)\geq \mathcal{C}_3 (|t|^{m-2}t-|s|^{m-2}s)(t-s),
		\ \forall \ t,s \in \mathbb{R}.\nonumber
	\end{flalign}
	\begin{remark}	Hypothesis \eqref{g_dois} for $m=q/(q-1)$ is equivalent to
		\[
		\mathcal{C}^{1-q}_2 |t|^{q-2}t  \leq g(t) \leq \mathcal{C}_1 ^{1-q} |t|^{q-2}t, \ \forall \, t\in \mathbb{R}.
		\] 
		In particular, there exist positive constants $c_0, \, \mathcal{C}_0$ such that
		\begin{align}
			&|g(t)| \leq c_0 |t|^{q-1},\ \forall t\in \mathbb{R},\label{g_cresci}\\
			&|g^{-1}(t)| \leq\mathcal{C}_0 |t|^{m-1},\ \forall t\in \mathbb{R}.\label{h_cresci}
		\end{align}
		The constant $\mathcal{C}_2 / \mathcal{C}_1 \geq 1$ measures how close the nonlinearity $h(t)=g^{-1}(t)$ is to the power $|t|^{m-2}t.$ 
  Condition \eqref{g_tres} is crucial in our concentration-compactness argument to prove that
		\begin{equation*}
			b(x)H\left( (b(x))^{-1} L(u_k) \right) \rightarrow b(x)H\left( (b(x))^{-1} L( u) \right) \text{ a.e. in }\mathbb{R}^N,
		\end{equation*}
		provided that $u_k \rightharpoonup u$ in $\mathcal{D}^{2,m}_{\rad } (\mathbb{R}^N )$ or $W^{2,m}_V(\mathbb{R}^N)$ $($see Proposition \ref{p_aeconv} $)$.
	\end{remark}
	
	\begin{remark}
		If $(u,v)$ is a weak solution for \eqref{S}, then $(u,v)$ is a strong solution for \eqref{S}. Moreover, if $u$ is a weak solution for \eqref{P} and $v = h( (b(x))^{-1} L(u) )$ belongs to either $W_V^{2,m}(\mathbb{R}^N)$ (for the subcritical case) or $\mathcal{D}^{2,l}(\mathbb{R}^N)$ (for the critical case), then $ I(u) = J(u,v) \quad \mbox{and} \quad  I'(u) \cdot \varphi = J'(u,v)\cdot (\varphi, 0),$ by the identity $G(h (t)) = t h(t) - H(t)$. Thus, $(u,v)$ is a weak solution for \eqref{S} if, and only if $v = h( (b(x))^{-1} L(u) )$ with $u$ a weak solution for \eqref{P}. See Section \ref{s_regularity}, for further regularity properties for \eqref{P} and \eqref{S}.
	\end{remark}
	\begin{remark}
		Theorems \ref{p_regsub} and \ref{p_regcrit} (see below) ensure conditions to establish that $(u,v)$ is a weak solution for \eqref{S}. In other terms, those regularity results ascertain whenever the above definitions of weak solutions for \eqref{P} and \eqref{S} are equivalent. In particular, if $(u,v)$ are solutions under the conditions of Theorems \ref{p_regsub} and \ref{p_regcrit} (see below) we get that
		\begin{equation}\label{condicao_poho_weak}
			2\int _{\mathbb{R}^N}\left\langle\nabla u, \nabla v \right\rangle \dx = \int _{\mathbb{R}^N}v(b(x)g(v) - V(x) u) + u(f(x,u) - V(x)v)\dx.
		\end{equation}
	\end{remark}

\medskip 
 
 \begin{flushleft}
 \textit{Notation:}\\
 A function $\xi (x,t),$ $x\in \mathbb{R}^N,$ $t \in \mathbb{R}$ is $\mathbb{Z}^N$--periodic, when $\xi (x+y,t) = \xi (x,t),$ for all $y \in \mathbb{Z}^N.$\\
  $\mathcal{X}_A$ is the characteristic function of $A\subset \mathbb{R}^N.$\\
  $B_R(x_0) = \{ x \in \mathbb{R}^N : |x-x_0| < R\}$  and $B_R =B_R(0),$ \\
  $C_R(x_0,t_0) = (B_{R+1}(x_0)\setminus B_R(x_0)) \times [0,t_0],$ \\
  $|A|$ denotes the Lebesgue measure of $A \subset \mathbb{R}^N$,\\
$\dashint _{B_R(x_0)}  u  \dx$ denotes the average of $u$ over the ball.
 \end{flushleft}

	%
	%
	
	\subsection{Subcritical case: (\ref{HYPsub}).}\label{s_subcritical}  
	Motivated by \cite{edersonGSH,dooH}, a natural question arises: what are the optimal conditions on the potential and nonlinearities to obtain the existence of solutions of \eqref{S}?  We expose our contribution to elucidating this question for the superlinear case.  Let $f: \mathbb{R}^N \times \mathbb{R} \rightarrow \mathbb{R}$ be a Carath\'{e}odory function and assume that
	\begin{flalign}\tag{$f_1$}\label{bem_def}
		\forall \, \varepsilon >0,\ \exists \, C_\varepsilon >0, \ p_\varepsilon \in (m,m_\ast)\text{ s.t. }\\|f(x,t)| \leq \varepsilon|t|^{m-1} + C_{\varepsilon} |t| ^{p_{\varepsilon } -1}\text{ a.e. } x \in \mathbb{R}^N, \ \forall \,  t \in \mathbb{R},\nonumber
	\end{flalign}
	where $m_\ast=mN/(N-2m)$ if $2m <N$ and $m_\ast=+\infty$ if $2m =N$.
	
	Suppose the following  Ambrosetti-Rabinowitz type condition: 
	\begin{flalign}\tag{$f_2$}\label{A-R}
		\exists \, \mu  > m \frac{\mathcal{C}_2}{\mathcal{C}_1} \; \text{  s.t.  } \; \mu F(x,t) \leq f(x,t)t \quad \text{a.e.}\quad x \in \mathbb{R}^N, \  \forall \, t \in \mathbb{R}. 
	\end{flalign}
	Usually in \eqref{A-R} is also required $F(x,t) >0$ but here it can change sign provided 
	\begin{flalign}\tag{$f_3$}\label{posi_algum}
		\exists\,  R, t_0>0, \ x_0 \in \mathbb{R}^N\text{ s.t. }\\|B_R|\inf _{B_R(x_0)} F(x,t_0) + |B_{R+1}\setminus B_R| \inf_{C_R(x_0,t_0)} F(x,t) >0.\nonumber
	\end{flalign}
	The following condition provide a regularity result to \eqref{S}:
	\begin{flalign}\tag{$f_4$}\label{subDACERTO}
		\exists \ \mathcal{C}_0>0, \, p \in (m, m_\ast)\text{ s.t. }|f(x,t)| \leq\mathcal{C}_0 |t|^{p-1}\ \text{ for a.e. }x\in \mathbb{R}^N,\ \forall t\in \mathbb{R}.
	\end{flalign}
	Since $m=q/(q-1)$, condition  $m<p$ correspond to $(p-1)(q-1)>1$ and  $ p < m_\ast $ is equivalent to \eqref{HYPsub}.
	It is required the following hypotheses on the potential $V(x):$
	\begin{flalign}\tag{$V_1$}\label{V_sirakov}
		\left\{ \
		\begin{aligned}
			&V(x) \in L ^{\nu} _{\loca}(\mathbb{R}^N),\ \nu =\max\{m, N/2\},\ V(x)\geq 0\text{ a.e. } x \in \mathbb{R}^N \text{ and}\\
			&\mathcal{C}_{V} := \inf _{u \in C ^\infty _0 ( \mathbb{R} ^N),\ \|u\|_m = 1} \int _{\mathbb{R}^N} |\Delta u|^m + |V(x)u |^m\dx>0.
		\end{aligned}
		\right.
	\end{flalign}
	\begin{flalign}\tag{$V_2$}\label{V_pesocerto}
		\left( \dashint _{B_R(x_0)}|V(x)|^{\nu } \dx \right)^{\frac{1}{\nu}}\leq \mathcal{C}\left( \dashint _{B_R(x_0)} |V(x)| \dx\right), \forall \, B_R(x_0)\subset \mathbb{R}^N,\nonumber
	\end{flalign}
	where $ \mathcal{C}= \mathcal{C}(V,\nu) >0 .$

 \subsubsection{Existence results:}
	\begin{theorem}\label{GS_subcrit}
		Assume \eqref{V_sirakov}, \eqref{V_pesocerto}, \eqref{bem_def}--\eqref{posi_algum}, \eqref{g_dois}, \eqref{g_tres}, and that $b(x),$ $V(x)$ $f(x,t)$ are $\mathbb{Z}^N$--periodic. Then,
   \begin{flushleft}
$ \mathrm{(i)}$ Eq.~\eqref{P} has a strong solution $u \in W^{2,m}_V(\mathbb{R}^N)$. Moreover, taking 
		\begin{equation}\label{Jaboti}
			v = h \left((b(x))^{-1}(L(u)\right) \in L^q(\mathbb{R}^N)\cap W_{\loca}^{2, q}(\mathbb{R}^N),
		\end{equation}
		the pair $(u,v)$ is a strong solution for Syst.~\eqref{S}.\\
$\mathrm{(ii)}$  If $g(t) \equiv |t|^{q-2}t,$ then  $u$ is  a ground state solution for \eqref{P}.\\
$\mathrm{(iii)}$ If  \eqref{subDACERTO} holds and $V(x) \in L^{\infty}(\mathbb{R}^N)$, then $(u,v) \in W^{2,m}_V(\mathbb{R}^N) \times W^{2,l}_V(\mathbb{R}^N)$ is a ground state for \eqref{S}.
  \end{flushleft}
	\end{theorem}

	\begin{remark} 
		Hypothesis \eqref{V_sirakov} gives the continuous embedding of $W^{2,m}_V (\mathbb{R}^N)$ into a  Lebesgue space (Proposition \ref{p_normasirakov}). To prove compactness for $I,$ we use its translation invariance combined with a concentration-compactness principle (Sect. \ref{s_cc}).
	\end{remark}

	\begin{remark}
		Nonnegative and locally integrable potentials satisfying \eqref{V_pesocerto} are known as the reverse H\"{o}lder class, and many authors studied it to discuss a priori $L ^\theta$ estimates for elliptic operators (see \cite{brama,shen} for further discussion).
	\end{remark}

	\begin{remark} Alternatively of Theorem~\ref{GS_subcrit}, suppose \eqref{V_sirakov}, \eqref{V_pesocerto}, \eqref{bem_def}--\eqref{posi_algum}, \eqref{g_dois}, \eqref{g_tres} and $t \mapsto f(x,t)|t|^{-1}$ is a strictly increasing.  
		Then, Eq.~\eqref{P} has a ground state solution  $u$. If in addition \eqref{subDACERTO} holds and $V(x) \in L^{\infty}(\mathbb{R}^N)$, then $(u,v) $ is a ground state for \eqref{S}, where $v$ was given in \eqref{Jaboti}.
	\end{remark}
	
	%
	%

	\subsubsection{Regularity results: } Weak solutions for \eqref{S} are indeed strong solutions.
	\begin{theorem}\label{p_regsub}Suppose $V(x) \in L^\infty _{\loca } (\mathbb{R}^N),$ \eqref{V_sirakov}, \eqref{V_pesocerto}, \eqref{bem_def}--\eqref{posi_algum} and \eqref{g_dois}. Let $u \in W^{2,m}_V (\mathbb{R}^N)$ be a weak solution for \eqref{P}, and $v := h \left((b(x))^{-1} L (u) \right).$ Then $v\in L^q(\mathbb{R}^N)\cap W_{\loca}^{2, \theta}(\mathbb{R}^N),$ for all $\theta \geq 1,$ and the pair $(u,v)$ is a strong solution for Syst.~\eqref{S}. Moreover, if \eqref{subDACERTO} holds and $V(x) \in L^{\infty}(\mathbb{R}^N)$, then $v \in W_V^{2,l}(\mathbb{R}^N).$
	\end{theorem}

	Schauder estimates and Theorem \ref{p_regsub} implies $C^2$ regularity for the solutions of Theorem \ref{GS_subcrit}.

	\begin{corollary}If $f(x,t),$ $b(x)g(t)\in C^{0,\alpha }_{\loca } (\mathbb{R}^N\times \mathbb{R}),$ $V(x) \in C^{0,\alpha }_{\loca } (\mathbb{R}^N) $ for some $0<\alpha <1,$ then solutions of Theorem \ref{GS_subcrit} belongs to $(C^{2,\alpha}_{\loca} (\mathbb{R}^N ) )^2.$
	\end{corollary}
	The novelty of Theorem \ref{p_regsub} also resides in the arguments used in its proof. Thanks to the regularity type results in  \cite{veryweak} on very weak solutions, our argument reduces the analysis of \cite{edersonGSH} to a direct approach involving a bootstrap argument.

	%
	%

	\subsection{Critical case: (\ref{HYPcrit})}\label{s_critical} To study this case, it is required that the nonlinearities are \textit{self-similar}. Precisely,
	\begin{flalign}\tag{$g_3$}\label{g_quatro}
		h (t) = 2^{\frac{N-mN}{m}j} h (2 ^{\frac{N}{m}j} t),\, \forall \, t \in \mathbb{R},\ j \in \mathbb{Z}.
	\end{flalign}
	One can see that if $h(t)$ satisfies \eqref{g_quatro} then
	\begin{equation*}
		H(t) = 2^{-Nj}H(2 ^{\frac{N}{m}j} t),\, \forall \, t \in \mathbb{R},\ j \in \mathbb{Z}.
	\end{equation*}
	We ask $f(x,t)\equiv f(t) \in C(\mathbb{R})$ and
	\begin{flalign}\tag{$f_1 ^\ast$}\label{H_um}
		f(t)=2 ^{\left(\frac{N-2m}{m}-N\right)j} f\left(2 ^{\frac{N-2m}{m}j} t \right),\, \forall \, t \in \mathbb{R},\ j \in \mathbb{Z}.
	\end{flalign}
	If $f(t)$ satisfies \eqref{H_um}, then \eqref{subDACERTO} holds and
	\begin{equation}\label{original_selfsimilar}
		F(t)=2 ^{-Nj} F\left(2 ^{\frac{N-2m}{m}j} t \right),\, \forall \, t \in \mathbb{R},\ j \in \mathbb{Z}.
	\end{equation}
	We describe the behavior of  \textit{self-similar} nonlinearities in Section \ref{s_selfsimilar}. 
	
	For the following results, we assume $b(x) \equiv 1,$ general oscillating nonlinearities, $V(x) \equiv V_\lambda(x) \equiv  -\lambda |x|^{-2},$ with $ \lambda \geq 0$ and
	\begin{equation*}
		\fbox{
			$
			L_\lambda(u) = -\Delta u -\lambda|x|^{-2} u.
			$
		}
	\end{equation*}
	
	\subsubsection{Existence results}
	
	\begin{theorem}\label{teo_mincrit}
		Assume \eqref{H_um}, \eqref{A-R} and that there exists $t_\ast > 0 $ such that $F(t_\ast) > 0$. 
		Then,  there exists a precompact minimizing sequence $( 2 ^{-\frac{N-2m}{m}j_k } u_k (2 ^{- j_k } \cdot ) ),$ $(j_k) \subset \mathbb{Z}$ for the infimum
		\begin{equation}\label{min}
			\Theta _{p,q,\lambda} = \inf \left\lbrace  \int_{\mathbb{R}^N}  |L_\lambda(u)|^m \dx : u \in  \mathcal{D}_\rad^{2,m} (\mathbb{R}^N),\ \int _{\mathbb{R}^N} F(u) \dx = 1 \right\rbrace,
		\end{equation}
		provided  $0 \leq \lambda < \Lambda _{N,m}^{1/m}$. Therefore, \eqref{min} is attained.
	\end{theorem}
	Using a Pohozaev type identity for Hamiltonian systems (see Theorem \ref{p_pohozaev_id}), we are able to prove the next qualitative propriety for the minimizer of \eqref{min}.
	\begin{theorem}\label{teo_mincrit2}
		Under the assumptions of Theorem \ref{teo_mincrit}, if $w$ is a minimizer of \eqref{min}, then for some $\beta>0,$  $u=w(\cdot / \beta )$ is a ground state (strong) solution of Eq.~\eqref{P}. Moreover,  taking $v = | L_\lambda(u)|^{m-2}(L_\lambda(u)),$ the pair $(u,v) \in \mathcal{D}^{2,m}_{\rad}(\mathbb{R}^N) \times \mathcal{D}^{2,l}_{\rad}(\mathbb{R}^N)$ is a ground state (strong) solution of Syst.~\eqref{S} provided $g(t) \equiv |t|^{q-2}t,$ $b(x) \equiv 1$ and $V(x) = - \lambda |x|^{-2}.$
	\end{theorem}
	Theorem \ref{teo_mincrit} deals with the existence of ground states for \eqref{S} using constrained minimization requiring $g(t) \equiv |t|^{q-2}t,$ see \cite{lionscompcase1,lionscompcase2}. Next, we deal with a more general class of nonlinearities $g(t)$ by using minimax techniques.
	\begin{theorem}\label{GS_crit}
		Assume \eqref{H_um}, \eqref{A-R}, \eqref{g_dois}--\eqref{g_quatro} and that there exists $t_\ast > 0 $ such that $F(t_\ast) > 0.$ Then Eq. \eqref{P} has a ground state (strong) solution, provided 
		$b(x)\equiv 1,$ $V(x) \equiv  -\lambda |x|^{-2}$ and $0 \leq \lambda < \Lambda _{N,m}^{1/m}$. Moreover, setting $v = h \left(L_\lambda(u) \right),$ the pair $(u,v) \in \mathcal{D}^{2,m}_{\rad}(\mathbb{R}^N) \times \mathcal{D}^{2,l}_{\rad}(\mathbb{R}^N)$ is a ground state (strong) solution for Syst.~\eqref{S}.
	\end{theorem}
	In the context of Theorem \ref{GS_crit}, we point out that under additional assumptions, if $\lambda = 0,$ we prove in Theorem \ref{c_regularity_crit_0} that $(u,v) \in C^{2,\alpha}_{\loca} (\mathbb{R}^N) \times C^{2,\alpha}_{\loca} (\mathbb{R}^N).$ Also, Theorem \ref{GS_crit} provides the existence of ground states solutions for \eqref{S} for a class of oscillating critical nonlinearities $f(t)$ and $g(t)$ and singular potentials. To the authors' knowledge, Theorem \ref{GS_crit} is the first result dealing with Hamiltonian systems involving general nonlinearities with critical growth since J. Serrin and H. Zou \cite{serrinzou}, whereby topological methods were proved the existence of radial solutions requiring $C^2$ regularity.
	
	%
	%

	\subsubsection{Regularity results} In a similar fashion as made for problems with subcritical nonlinearities, we also have regularity results for problems involving critical growth by assuming,
	\begin{flalign}\tag{$f_2 ^\ast$}\label{constant_ss}
		\exists \, \mathcal{C}_{\ast}>0\text{ s.t. }|f(x,t)| \leq\mathcal{C}_{\ast} |t|^{p-1},\ \text{ for a.e. }x \in \mathbb{R}^N,\ \forall t\in \mathbb{R}.
	\end{flalign}
	
	\begin{theorem}\label{p_regcrit}
		Suppose \eqref{g_dois}, \eqref{constant_ss} \eqref{g_cresci}, \eqref{h_cresci}, $0<\beta \leq b(|x|)\in L^\infty (\mathbb{R}^N),$ $V(x) \equiv  -\lambda |x|^{-2}$ and $0 \leq \lambda < \Lambda _{N,m}^{1/m}.$ Let $u \in \mathcal{D}^{2,m}_\rad (\mathbb{R}^N)$ be a weak solution for Eq.~\eqref{P}, and take $v := h \left( (b(|x|))^{-1} L_\lambda (u) \right).$ Then $v\in \mathcal{D}^{2,l}_\rad (\mathbb{R}^N)$ and the pair $(u,v)$ is a radial strong solution for Syst.~\eqref{S}.
	\end{theorem}
	Theorem \ref{p_regcrit} explores the fact that solutions lie in a radial Sobolev space and have some local regularity and decay at infinity (see Lemma \ref{l_radial}). However, in $x=0,$ one does not have such a priori regularity even in the case that $\lambda =0.$ In order to prove further smoothness in neighborhoods of zero, we have to improve and complement some results of \cite{hulshof,vandervorst_best} for general nonlinearities and Banach spaces. We mention that in \cite{hulshof,vandervorst_best}, the Hilbert structure of function spaces or the $q=2$ was crucial in their argument.
\begin{theorem}\label{c_regularity_crit_0}
Under the conditions of Theorem \ref{p_regcrit} and $\lambda = 0,$ we have $u,v \in L^p_{\loca } (\mathbb{R}^N),$ for all $p\geq 1,$ provided one of the following assumptions is satisfied:
\begin{flushleft}
$\mathrm{(i)}$ $p,q>2$;\\
$\mathrm{(ii)}$ $h(t)$ is strictly increasing, $q\leq 2$ and there exists $\hat{C}_\ast >0$ such that
		\begin{equation}\label{f_porbaixo}\tag{$\hat{f}_\ast$}
			|f(x,t)| \geq \hat{C}_\ast |t|^{p-1},\text{ for a.e. }x \in \mathbb{R}^N,\ \forall \ t\in \mathbb{R};
		\end{equation}
$\mathrm{(iii)}$ $f(|x|,t) \equiv \hat{b}(|x|) \hat{f}(t),$ $p \leq 2$ and interchange the assumptions of $f(|x|,t)$ by the ones made on $b(|x|)g(t),$ more precisely: $\hat{f}(t)$ is a nondecreasing homeomorphism with $\hat{h}(t) := \hat{f}^{-1}(t)$ satisfyng \eqref{g_dois}, \eqref{g_cresci}, \eqref{h_cresci}, where $\hat{f}(t)$ and $\hat{h}(t)$ replaces $g(t)$ and $h(t),$ respectively. $\hat{b}(|x|)$ belongs to $L^\infty (\mathbb{R}^N),$ with $0<\hat{\beta } \leq \hat{b}(x)$ and writing $\hat{g}(x,t) = b(x)g(t),$ the function $\hat{g}$ verifies \eqref{constant_ss}, with $f$ replacing $\hat{g}.$
\end{flushleft}
\end{theorem}
\begin{corollary}
If $ f(t),$ $g(t) \in C^{0,\alpha }_{\loca } (\mathbb{R}),$ for some $0<\alpha <1,$ then the solutions $(u,v)$ of Theorems \ref{teo_mincrit2} and \ref{GS_crit} belongs to $( C^{2,\alpha}_{\loca} (\mathbb{R}^N \setminus \{0\}) )^2.$ Moreover, if the conditions of Theorem \ref{c_regularity_crit_0} hold, then $(u,v)\in ( C^{2,\alpha}_{\loca} (\mathbb{R}^N) ) ^2.$
\end{corollary}

	%
	%
	\subsection{Nonexistence results} 
	Our following result is a Pohozaev identity for  \eqref{S}, which is a key to prove our general nonexistence results. Precisely, we show that some solutions of \eqref{S} must satisfy a specific integral identity, which also holds for smooth bounded domains (see \cite{Mitidieri-1993,vandervorst_prim,puccin-serrin}). In its proof, we use a “local-to-global” cutoff argument. This identity is crucial in our argument to control the energy level of the quasilinear term of the energy functional.
	%
	%
	\begin{theorem}\label{p_pohozaev_id}
		Let $(u,v) \in W_{\loca}^{2, q/(q-1)}(\mathbb{R}^N) \times W_{\loca}^{2, p/(p-1)}(\mathbb{R}^N )$ be a strong solution of \eqref{S} satisfying  \eqref{condicao_poho_weak}.  Suppose $0<\beta \leq b(x)\in L^\infty (\mathbb{R}^N)\cap C^1 (\mathbb{R}^N)$ and $V(x)\in C^1(\mathbb{R}^N \setminus \mathcal{O}),$ where $\mathcal{O}$ is a finite set.	Then it holds
		\begin{multline}
			N \int_{\mathbb{R}^N} F(x,u) + b(x) G(v)\dx +\sum _{i=1}^N \int _{\mathbb{R}^N}x_i F_{x_i} (x,u) \dx + \int _{\mathbb{R}^N} \left\langle  x,\nabla b (x) \right\rangle G(v) \dx=\\
			\frac{N-2}{2}\int_{\mathbb{R}^N}  uf(x,u) + vb(x)g(v)\dx + \int_{\mathbb{R}^N} \left[ 2 V(x) + \left\langle x, \nabla V(x) \right\rangle \right] uv \dx,\label{SIS_poho}
		\end{multline}
		provided the following terms belong to $ L^1(\mathbb{R}^N )$,
		\begin{equation}
			F(x,u),\ b(x)G(v),\ \sum _{i=1}^N x_i F_{x_i}(x,u),\  \left\langle \nabla b(x),x\right\rangle G(v),\ V(x)uv,\ \left\langle \nabla V(x),x \right\rangle uv,\ uf(x,u),\ vb(x)g(v).\label{condicao_poho}
		\end{equation}
		Moreover, if $g(t)$ is a nondecreasing homeomorphism and 
		\begin{equation*}
			v= h \left( (b(x))^{-1}L (u) \right) \in W_{\loca}^{2, p/(p-1)}(\mathbb{R}^N ),	
		\end{equation*}
		then the following identity holds 
		\begin{multline}
			N\int_{\mathbb{R}^N }b(x) H\left((b(x))^{-1}L(u)\right)\dx+ \int_{\mathbb{R}^N }\left\langle  x, \nabla b(x) \right\rangle H\left((b(x))^{-1}L(u)\right)\dx \\ 
			+ \int_{\mathbb{R}^N} h \left((b(x))^{-1} L(u)\right)\left[ 2 \Delta u + \left\langle x,\nabla V(x) \right\rangle u  -(b(x))^{-1}\left\langle x,\nabla b(x) \right\rangle L(u)\right]\dx \\ = N \int _{\mathbb{R}^N}F(x,u) \dx + \sum _{i=1}^N \int _{\mathbb{R}^N}x_i  F_{x_i}(x,u) \dx.\label{EQ_poho}
		\end{multline}
	\end{theorem}
We emphasize that Theorem \ref{p_pohozaev_id} gives a partial answer for the Lane-Endem conjecture  for the case of weak solutions (see \cite{zouh,souplet}), which we state as follows:
\begin{corollary}\label{c_leconject} 
 \begin{flushleft}
$\mathrm{(i)}$ Suppose $N \geq 3$ and $(p,q)$ satisfies \eqref{HYPsuper} or \eqref{HYPcrit}. Let 
     $(u,v) \in W^{2, m}(\mathbb{R}^N) \times W^{2, l}(\mathbb{R}^N )$ be a nonnegative 
  weak solution of the following system
		\begin{equation*}
			\left\{\begin{array}{lll}
				-\Delta v+ v = |u|^{p-2}u \quad &\mbox{in}\quad \mathbb{R}^N,\\
				-\Delta u+ u = |v|^{q-2}v \quad &\mbox{in}\quad \mathbb{R}^N.
			\end{array}\right.
		\end{equation*}
		 Then $u=v=0.$\\
$\mathrm{(ii)}$  Suppose $N \geq 3$ and $(p,q)$ satisfies \eqref{HYPsub}. 
Let $ (u,v) \in  \mathcal{D}^{2, m}(\mathbb{R}^N) \times  \mathcal{D}^{2, l}(\mathbb{R}^N )
		$ be  a weak solution of
		\begin{equation*}
			\left\{\begin{array}{lll}
				-\Delta v = |u|^{p-2}u \quad &\mbox{in}\quad \mathbb{R}^N,\\
				-\Delta u =|v|^{q-2}v \quad &\mbox{in}\quad \mathbb{R}^N.
			\end{array}\right.
		\end{equation*}
		Then $u=v=0$.
\end{flushleft} 
	\end{corollary}

	%
	%

	\subsection{Additional remarks on the hypotheses} \ 
	
\noindent$\mathrm{(i)}$ Condition \eqref{bem_def} is satisfied by any $f: \mathbb{R}^N \times \mathbb{R} \rightarrow \mathbb{R}$  Carath\'{e}odory function such that $\lim_{t\rightarrow 0}f(x,t)/|t|^{m-1}=0,$ uniform in $x$, and such that
	\begin{equation*}
		|f(x,t)|\leq C(1+|t|^{\varrho(t)-1}) \quad \text{  almost everywhere }  \mathbb{R}^N, \; \forall \,  t \in \mathbb{R},
	\end{equation*}
	where $\varrho(t) \in L^\infty ( \mathbb{R}),$ $m<\inf_{t \in \mathbb{R}}\varrho (t)\leq  \sup_{t \in \mathbb{R}}\varrho(t)<m_\ast$ and $\inf _{x\in[\delta,1]}\varrho(x) \geq \sup _{x \in [1,\infty)}\varrho(x),$ for any $\delta \in (0,1).$ For this case, in the above conditions is implicit that $(\varrho(t) - 1)(q-1)>1,$ and 
	\begin{equation*}
		\frac{1}{\varrho(t)} + \frac{1}{q} > 1 - \frac{2}{N},\ \varrho(t),q >1,\ N \geq 3.
	\end{equation*}
	\noindent$\mathrm{(ii)}$ Theorems \ref{p_regsub} improves and extends \cite[Theorem A.1]{edersonGSH} for more general nonlinearities and potentials. Besides, if we suppose $f(x,t) \equiv f(t)$ is a homeomorphism additionally to the assumptions of Theorem \ref{p_regsub}, then the same argument of \cite[Theorem A.1]{edersonGSH} gives
	\[
	\begin{aligned}
		u \in W^{2,\theta}(\mathbb{R}^N), \text{ for } \max\{1,1/(p-1) \} < \theta < \infty,\\
		v \in W^{2,\hat{\theta}} (\mathbb{R}^N), \text{ for }  \max\{1,1/(q-1) \} < \hat{ \theta }< \infty.
	\end{aligned}
	\]
	In particular, $|u(x)|,$ $|v(x)|,$ $|\nabla u(x)|,$ $|\nabla v(x)|\rightarrow 0,$ as $|x|\rightarrow \infty.$
	
	\noindent$\mathrm{(iii)}$ Suppose $g(t) \in C^{2,\alpha}_{\loca} (\mathbb{R}^N)$ and the hypotheses of Theorem \ref{p_regsub}, then $u \in W^{4,\theta}(\mathbb{R}^N)\cap C^{4,\alpha }_{\loca} (\mathbb{R}^N),$ for some $\theta > 1.$ Similarly, under the assumptions of Theorem \ref{p_regcrit}, we have  $u \in  \mathcal{D}^{4,\theta }_\rad (\mathbb{R}^N)\cap C^{4,\alpha }_{\loca} (\mathbb{R}^N).$	
	
	\noindent$\mathrm{(iv)}$ If there exists $t_0 > 0 $ such that $F(t_0) > 0,$ then the infimum in \eqref{min} is positive. 
	Indeed,  for $ R>0,$ take  $v _R \in C^\infty _0 (\mathbb{R},[0,t_0])$ such that $v _R(t)=t_0,$ if $|t|\leq R,$ and $v _R(t)=0,$ if $|t|>R+1.$ 
	Defining $\varphi_R(x) := v_R(|x|),$ we have $\varphi_R \in \mathcal{D}^{2,m}_{\rad } (\mathbb{R}^N ).$ Moreover, 
	\begin{equation*}
		\int_{\mathbb{R}^N} F(\varphi_R ) \dx \geq F(t_0)|B_R|-|B_{R+1}\setminus B_R|\left(\max_{t \in [0,t_0]} |F(t)|\right)>0,
	\end{equation*}
	if $R$ is large.  Thus, taking a suitable $\sigma >0$, 
	$
	\int _{\mathbb{R}^N} F(\varphi_R (\cdot / \sigma )) \dx = 1.
	$
	\begin{example}
		Our approach includes the following potentials and nonlinearities:
		
		\noindent$\mathrm{(i)}$ If  $V(x)\in L ^\infty(\mathbb{R}^N)$ satisfies \eqref{V_pesocerto} is $\mathbb{Z}^N$--periodic and strictly positive.
		
		\noindent$\mathrm{(ii)}$ In Example \ref{ex_ss_g}, we construct a function $g$ satisfying \eqref{g_dois}--\eqref{g_quatro}.
		
		\noindent$\mathrm{(iii)}$ Also, if \eqref{HYPsub} holds, then $f(x,t)=k(x)\left[ \varrho'(t) (\ln |t| t) + \varrho (t)\right] |t|^{\varrho(t) -2}t$ with $ f(x,0)\equiv 0,$ satisfies \eqref{bem_def}, \eqref{A-R} and \eqref{posi_algum}, for $\varrho(t) = \delta_1 \sin\left( \ln( |\ln|t| |)\right)+\delta_2,$ $0 < \inf _{x \in \mathbb{R}^3} k(x) \leq   k(x)\in C(\mathbb{R})\cap L^\infty(\mathbb{R}^N),$  and suitable $\delta_1, \delta_2>0.$
		
		\noindent$\mathrm{(iv)}$ The nonlinearity $f(t) = (c \delta \cos(\delta \ln |t| ) +p)( \exp\{c(\sin(\delta \ln |t|) +1)\} )|t|^{p-2}t, $ $f(0):=0,$ where $\delta = 2 \pi m /(\ln2 (N-2m) )$ fulfills the conditions of Theorem~\ref{GS_crit}, whenever \eqref{HYPcrit} and $\delta c+\mu\leq p$ hold.
	\end{example}
	
	%
	%
	
	\subsection{Outline}
	Sect. \ref{s_preli}, is dedicated to developing our variational framework. In Sect.~\ref{s_regularity}, we prove some regularity results and a Pohozaev-type identity, which allows us to obtain some nonexistence results. In Sect.~\ref{s_cc}, we study a concentration-compactness employing profile decomposition type result and the almost everywhere convergence of the quasilinear term of $I$. Sections \ref{s_t_mincrit} and \ref{s_t_proof} are dedicated to the proof of our main results.
	
	\section{Preliminaries results}\label{s_preli}
	
	\subsection{Variational Settings} Next, we present some properties of the spaces $W^{2,m}_V(\mathbb{R}^N)$ and  $\mathcal{D}_\rad^{2,m} (\mathbb{R}^N).$
	\begin{proposition}\label{p_normasirakov} $W^{2,m}_V(\mathbb{R}^N)$ is a well-defined Banach space and it is continuously embedded in $W^{2,m} (\mathbb{R}^N)$ provided that \eqref{V_sirakov} holds. Moreover,
		\begin{equation}\label{conversely}
			W^{2,m}_V(\mathbb{R}^N) \subset \left\lbrace u \in W^{2,m}(\mathbb{R}^N) : \int_{\mathbb{R}^N} |V(x)u|^m\dx < \infty \right\rbrace .
		\end{equation}
	\end{proposition}
	\begin{proof} Clearly
		\begin{equation}\label{para_provar}
			\int_{\mathbb{R}^N} | \Delta u |^m \dx\leq \int _{\mathbb{R}^N} |\Delta u|^m + |V(x)u |^m\dx, \ \forall \, u \in C^\infty _0( \mathbb{R}^N).
		\end{equation}
		Also by \eqref{V_sirakov} we have
		\begin{equation*}
			\mathcal{C}_V \int _{\mathbb{R}^N}|u|^m \dx\leq \int _{\mathbb{R}^N} |\Delta u|^m + |V(x)u |^m\dx, \ \forall \, u \in C^\infty _0( \mathbb{R}^N).
		\end{equation*}
		From \eqref{para_provar}, for any sequence $(\varphi _n) \subset C^\infty _0 (\mathbb{R}^N),$
		\begin{equation*}
			\|\varphi _k - \varphi _l \|^m_{2,m}\leq (1 + \mathcal{C}_V)\|\varphi _k - \varphi _l\|^m_V,\ \forall \, k\neq l.
		\end{equation*}
		Consequently $W^{2,m}_V(\mathbb{R}^N)$ is well defined with continuous embedding in $W^{2,m}(\mathbb{R}^N),$ and \eqref{conversely} holds by Fatou Lemma. \end{proof}
	\begin{corollary}
		Assume \eqref{V_sirakov} and $V(x) \in L^{\infty}_{\loca}(\mathbb{R}^N),$ then
		\begin{equation*}
			W^{2,m}_V(\mathbb{R}^N) = \left\lbrace u \in W^{2,m}(\mathbb{R}^N) : \int_{\mathbb{R}^N} |V(x)u|^m\dx < \infty \right\rbrace .
		\end{equation*}
		In particular, $W^{2,m}_V(\mathbb{R}^N) = W^{2,m}(\mathbb{R}^N)$ provided $V(x) \in L^{\infty}(\mathbb{R}^N).$
	\end{corollary}
	\begin{proof} 
		Let $u \in W^{2,m} (\mathbb{R}^N)$ such that $\int_{\mathbb{R}^N} |V(x)u|^m\dx < \infty.$ Assume first that $u$ has compact support. Let $(\varrho_k)$ a sequence of mollifiers and define $u_k = \varrho_k \ast u \in C^\infty_0(\mathbb{R}^N).$ We have that $u_k \rightarrow u$ and $\Delta u_k \rightarrow \Delta u$ in $L^m(\mathbb{R}^N).$ Since $K:=\supp(u) \cup \overline{B_1 }$ is a compact set,
		\begin{equation*}
			\int_{\mathbb{R}^N}|V(x)(u_k -u)|^m \dx\leq \|V(x)\|^m_{L^\infty(K)}\int_{\mathbb{R}^N} |u_k -u|^m\dx \rightarrow 0, \text{ when }k \rightarrow \infty.
		\end{equation*}
		For the general case, we consider a truncation function $\xi \in C_0 ^\infty (\mathbb{R}^N)$ such that $\xi\equiv 1$ in $B_1,$ $\xi\equiv 0$ on $\mathbb{R}^N \setminus B_2$ and $\|\Delta \xi\|_\infty, \| \nabla \xi\|_\infty \leq M.$ Setting $\xi _k (x) = \xi (x/k)$ and $u_k = \xi_k u \in W^{2,m} (\mathbb{R}^N)$ (which has compact support) we have $\|u_k - u\|_V \rightarrow 0,$ by Lebesgue's theorem.
	\end{proof}
	We consider the following norm on $W^{2,m}_V(\mathbb{R}^N),$
	\begin{equation*}
		[ \, u \, ]_V := \left[ \int_{\mathbb{R}^N}|L(u)|^m \dx \right] ^{1/m},
	\end{equation*}
	which gives the appropriated variational formulation for weak solutions of \eqref{P}.
	\begin{corollary}\label{l_norma_comparar_sub}
		The norms $[\ \cdot \ ]_V$ and $\| \cdot \| _V$ are equivalent in $W^{2,m}_V (\mathbb{R}^N),$ if \eqref{V_sirakov} and \eqref{V_pesocerto} hold.
	\end{corollary}
	\begin{proof}
		Since $V(x)$ belongs to the reverse H\"{o}lder class $B_{\nu},$ and $\nu \geq N/2,$ we may apply the results of \cite[Corollary 0.9]{shen} to obtain $C>0$ such that
		\begin{equation*}
			\int _{\mathbb{R}^N} |\Delta u|^{\theta} + |V(x)u |^\theta\dx\leq C \int _{\mathbb{R}^N}|L(u)|^\theta \dx,\ \forall \, u \in C^\infty _0 (\mathbb{R}^N),
		\end{equation*}
		for $1< \theta \leq \nu.$ In particular, we can take $\theta = m.$ On the other hand, using $|a+b|^p \leq 2^{p-1} (  |a| ^p + |b|^p  )$, we have $[ \ u \ ]^m_V \leq 2^{m-1} \| u\|^m _V.$
	\end{proof}
	To dealing with nonlinearities with critical growth, we consider the space $\mathcal{D}^{2,m}_{\rad }(\mathbb{R}^3)$ endowed with the norm $\| u\|_{\lambda} := \left[ \int _{\mathbb{R}^N} |L_\lambda (u)|^m \dx\right]^{1/m}.$
	\begin{proposition}\label{p_norma_comparar} The norm $\|  \cdot  \|_{\lambda}$ is equivalent to the usual norm in  $\mathcal{D}_\rad^{2,m} (\mathbb{R}^N),$ provided $0 \leq \lambda < \Lambda _{N,m}^{1/m}.$
	\end{proposition}
	\begin{proof}
		This result follows from \cite[Proposition 3.5 and Remark 3.7]{lpradialestimate}. Indeed,
		\begin{equation}\label{capetaineq}
			\int_{\mathbb{R}^N}|\Delta u|^m\dx \leq C \int_{\mathbb{R}^N}|L_\lambda (u) |^m\dx, \quad \forall \, u \in C_{0,\rad} ^\infty (\mathbb{R}^N),
		\end{equation}
		holds, if and only, 
		\begin{equation*}
			\Gamma_N ^2 + \lambda ^2 >\left|\frac{N}{m} - (\Gamma_N +2)\right| ^2,\text{ that is, }\lambda ^2 > \left( \frac{N}{m} - 2\right)^2 \left[1-\frac{2m\Gamma_N}{ N-2m}  \right],
		\end{equation*}
		$\Gamma_N =(N-2)/2.$ Thus \eqref{capetaineq} holds, since $N>2m, \; N-2m(1+\Gamma_N)<0.$
	\end{proof}
	\begin{corollary}\label{c_charaD} $\mathcal{D}_\rad^{2,m}(\mathbb{R}^N) = \{u \in L_\rad^p(\mathbb{R}^N) : L_\lambda (u) \in L_\rad^{m}(\mathbb{R}^N) \}$ if $0 \leq \lambda < \Lambda _{N,m}^{1/m}.$
	\end{corollary}
	
	\subsection{Existence of bounded Palais-Smale sequences at minimax levels}  
	Lets us consider the minimax level of the energy functional  $I$ associated to problem  \eqref{P},
	\begin{equation*}
		c(I)= \inf_{\gamma \in \Gamma_I } \sup_{t \geq 0 } I(\gamma (t)),
	\end{equation*}
	where  $E=\mathcal{D}_\rad^{2,m} (\mathbb{R}^N) $ or $E = W^{2,m}_V(\mathbb{R}^N)$ and
	$$
	\Gamma _{I} = \left\lbrace \gamma \in C([0, \infty ),  E) : \gamma(0)=0, \ \lim _{t \rightarrow \infty}I(\gamma (t)) = - \infty \right\rbrace .
	$$ 
	Now we prove that $I$ has a minimax geometry.
	
	\begin{proposition}\label{p_geoMP} $\mathrm{(i)}$ Assume that there exists $t_\ast> 0 $ such that $F(t_\ast)>0,\;   V(x) = -\lambda |x|^{-2}$ with $0 \leq \lambda < \Lambda _{N,m}^{1/m},$ and that \eqref{A-R}, \eqref{constant_ss}, \eqref{g_dois} hold. Then $0<c(I)<\infty$, and there exists a bounded sequence $(u_k)$ in $\mathcal{D}^{2,m}_\rad (\mathbb{R}^N)$ such that $I(u_k) \rightarrow c(I)$ and $I'(u_k) \rightarrow 0.$
		
		\noindent$\mathrm{(ii)}$ The same conclusion of (i) holds for $W_V^{2,m}(\mathbb{R}^N)$ if \eqref{HYPsub}, \eqref{V_sirakov}, \eqref{V_pesocerto}, \eqref{bem_def}, \eqref{A-R}, \eqref{posi_algum} and \eqref{g_dois} hold.
	\end{proposition}
	\begin{proof}
		Suppose \eqref{posi_algum} and $ \xi _R \in C^\infty _0 (\mathbb{R},[0,t_0]) $ such that $\xi_R(t) = t_0 $ if $|t| \leq R$ and $\xi _R(t)= 0$ if $|t| > R+1.$ Setting $v(x) = \xi_R(|x-x_0|),$ we have $v \in W_V^{2,m} (\mathbb{R}^N )$ and
		\begin{align*}
			\int_{\mathbb{R}^N} F(x,v ) \dx &= \int _{B_R (x_0)} F(x,t_0) \dx + \int_{B_{R+1}(x_0)\setminus B_R(x_0)} F(x,v ) \dx\\
			&\geq |B_R|\inf _{B_R(x_0)} F(x,t_0) + |B_{R+1}\setminus B_R|\inf_{C_R(x_0,t_0)} F(x,t) >0.
		\end{align*}
		Suppose $f(x,t) \equiv f(t)$ and $F(t_\ast)>0.$ Taking $\eta _R (x) = \xi _R (|x|),$ with $t_\ast$ instead of $t_0$ in the definition of $\xi _R.$ Then $\eta _R \in \mathcal{D}^{2,m}_\rad (\mathbb{R}^N)$ and
		\begin{align*}
			\int_{\mathbb{R}^N} F(\eta_R ) \dx &= \int _{B_R (x_0)} F(t_\ast) \dx + \int_{B_{R+1}(x_0)\setminus B_R(x_0)} F(\eta _R ) \dx\\
			&\geq F(t_\ast)|B_R|-|B_{R+1}\setminus B_R|\left(\max_{t \in [0,t_\ast]} |F(t)|\right).
		\end{align*}
		Thus, there exist $C_1,\ C_2>0$ such that for $R$ large,
		\begin{equation*}
			\int_{\mathbb{R}^N} F(\eta_R ) \dx\geq C_1 R^N - C_2 R^{N-1}>0.
		\end{equation*}
		By \eqref{constant_ss} and the Sobolev embedding $\mathcal{D}^{2,m}(\mathbb{R}^N) \hookrightarrow L^p (\mathbb{R}^N)$ we have,
		\begin{equation*}
			I(u) \geq \left((\mathcal{C}_1/m) \| b (x) \| _\infty ^{1-m}- \overline{C}\| L_\lambda (u) \|_m ^{p-m}\right)\| L_\lambda (u) \|^m_m,\ \forall \, u \in \mathcal{D}^{2,m}_\rad (\mathbb{R}^N),
		\end{equation*}
		for some $\overline{C}>0.$ Thus $I(u)\geq \delta >0,$ if $\| L_\lambda (u) \|_m=r$ is small. Notice that \eqref{A-R} is equivalent to $d/dt (F(t)t ^{-\mu}) \geq 0,$ for $t> 0.$ If $v(x)=\eta _R (x) = \xi _R (|x|),$ then
		\begin{equation*}
			\int_{\mathbb{R}^N}F(tv)\dx \geq t^\mu \int_{\mathbb{R}^N} F(v)\dx,\quad\text{whenever } t>1.
		\end{equation*}
		Hence, as $t \rightarrow \infty,$
		\begin{equation*}
			I(tv) \leq \mathcal{C}_2 t^m \int_{\mathbb{R}^N} |L_\lambda(v)|^m\dx - t^\mu \int_{\mathbb{R}^N} F(v)\dx \rightarrow - \infty.
		\end{equation*}
		Thus by the Ekeland Variational Principle, there exists $(u_k)\subset \mathcal{D}^{2,m}_\rad (\mathbb{R}^N)$ such that $I(u_k) \rightarrow c(I)$ and $I'(u_k) \rightarrow0.$ By \eqref{g_dois}, for large $k,$ we have
		\begin{equation}\label{psb_final}
			\begin{aligned}
				c(I)+1+\| L_\lambda (u_ k) \|_m  	&\geq I(u_k) - \frac{1}{\mu} I'(u_k) \cdot u_k \\
				&\geq \left(\frac{\mathcal{C}_1}{m} - \frac{\mathcal{C}_2}{ \mu }\right) \| b(x) \|_\infty ^{1-m} \|L_\lambda (u_ k)\|_{m} ^m,		
			\end{aligned}
		\end{equation}
		which implies that $(u_k)$ is bounded in $\mathcal{D}^{2,m}_\rad (\mathbb{R}^N)$ and (i) holds.
		
		Using \eqref{bem_def} and $W_V^{2,m}(\mathbb{R}^N) \hookrightarrow L^{\theta} (\mathbb{R}^N),$ $m \leq \theta <p,$ we have
		\begin{equation*}
			I(u) \geq \left( (\mathcal{C}_1/m) \| b (x) \| _\infty ^{1-m}  - \varepsilon C_0- \overline{C}_\varepsilon \|L(u)  \|_m ^{p_\varepsilon-m} \right)\|L(u) \|_m ^m,\ \forall \, u \in W_V^{2,m}(\mathbb{R}^N),
		\end{equation*}
		for some $\overline{C}_0 > 0. $ Consequently, taking $\varepsilon>0$ small enough, $I(u)\geq \delta >0,$ if $\| L (u) \|_m=r$ is small. Also, from the fact that $d/dt (F(x,t)t ^{-\mu}) \geq 0,$ for $t> 0,$ is equivalent to \eqref{A-R}, we can take $v(x) = \xi_R(|x-x_0|),$ to obtain
		\begin{equation*}
			I(tv) \leq (\mathcal{C}_2/m) \beta ^{1-m} t^m \int_{\mathbb{R}^N} |L(v)|^m\dx - t^\mu \int_{\mathbb{R}^N} F(x,v)\dx \rightarrow - \infty,\text{ as }t \rightarrow \infty.
		\end{equation*}
		Therefore, there exists $(u_k)\subset W_V^{2,m} (\mathbb{R}^N)$ such that $I(u_k) \rightarrow c(I)$ and $I'(u_k) \rightarrow0.$ Once again by \eqref{g_dois}, we can argue as in \eqref{psb_final} (replacing $L_\lambda$ by $L$) to conclude that $(u_k)$ is bounded in $W_V^{2,m} (\mathbb{R}^N)$ and (ii) holds.
	\end{proof}
	
	\subsection{Self-similar functions and some properties}\label{s_selfsimilar}
	\begin{definition} A function  $F \in C(\mathbb{R})$ 
		is self-similar with factor $\gamma >1$ and power $\sigma>0$ if 
		\begin{equation}\label{selfsimilar}
			F(t)=\gamma ^{-Nj} F(\gamma ^{\sigma j} t ), \;  \forall j \in \mathbb{Z}, \; t \in \mathbb{R}.
		\end{equation}
	\end{definition}
	
	\begin{proposition} 
		\noindent$\mathrm{(i)}$ A  continuously differentiable functions  $F$ is self-similar with factor $\gamma>1$ and power $\sigma>0$ if and only if 
		$	F'(t)=\gamma ^{(\sigma -N)j} F'\left(\gamma ^{\sigma j}  t \right), \; \forall j \in \mathbb{Z}, \ t \in \mathbb{R}.$
		Consequently,  $f \in C(\mathbb{R})$ is self-similar if its primitive satisfies \eqref{selfsimilar}.
		
		\noindent$\mathrm{(ii)}$ Self-similar functions behave like fractals; that is, they are completely determined once is known their value in some interval. 
		Let $I_j = [\gamma ^{\sigma j}, \gamma ^{\sigma (j+1)} ),$ $j \in \mathbb{Z}$ and $F_0: I_0 \rightarrow \mathbb{R}$ be a continuous function such that $\lim _{t \rightarrow \gamma ^\sigma }F_0 (t) = \gamma ^N F_0 (1)$. 
		If we define $F(t) = \gamma ^{Nj} F_0(\gamma ^{-\sigma j} t)$ if $t\in I_j$ for $j \in \mathbb{Z}$, we can see that $F$ is a self-similar function.
		
		\noindent$\mathrm{(iii)}$ Assume that $2m < N$ and $ F(t)$ is self-similar with factor $\gamma>1$ and power $\sigma =(N-2m)/m$. Then, for $m_\ast = mN/(N-2m),$ there exists $C>0$ such that $|F(t)| \leq C |t|^{m_\ast}, \;  \forall \, t\in \mathbb{R}.$ Moreover, if $F \in C^{k} (\mathbb{R}),$ then $|F^{(k)}(t)| \leq |t|^{m_\ast-k}, \; \forall \, t \in \mathbb{R},$ provided that $m_\ast > k.$
		
		\noindent$\mathrm{(iv)}$ Under the assumptions of the item  \textbf{(iii)}, $m_\ast = mN/(N-2m), \; u \in L^{m_\ast} (\mathbb{R}^N)$ and $ j \in \mathbb{Z}$, we have
		\begin{equation*}
			\int _{\mathbb{R}^N} F \left(\gamma ^{\frac{N-2m}{m}j} u(\gamma ^j x )\right) \dx = \int _{\mathbb{R}^N} F (u) \dx.
		\end{equation*}
		
		\noindent$\mathrm{(v)}$ Let $F(t)$ be a locally Lipschitz function and assume  the hypotheses of  item (iii),  then for  $a_1, \ldots, a_{M} \in \mathbb{R}$ there is $C=C(M)>0$ such that
		\begin{equation*}
			\left|F\left(\sum _{n=1} ^{M} a_n \right) - \sum_{n=1}^{M}F(a_n )\right| \leq C(M) \sum _{m \neq n \in \{1,\ldots,M\}} |a_n| ^{m_\ast -1} |a_m|.
		\end{equation*}
	\end{proposition}
	\begin{proof} The proof  is adapted from \cite[Lemma 6.4]{paper1} replacing the power $\sigma = (N-2s)/2$ by $\sigma = (N-2m)/m.$
	\end{proof}

	This class of  self-similar functions was introduced in \cite{tintass,tintabook} and later considered in \cite{paper1} to investigate a class of  fractional Schrodinger equations. They were also considered in \cite{selfsimilar} and are closely related to the notion of fractals. We follow the idea of \cite[Section 6]{paper1} to show that weak profile decomposition results like Theorem \ref{teo_profcrit} motivate extensions of the notion of nonlinearities with critical Sobolev growth and the definition of self-similar functions. As it can be seen in the next examples, self-similar functions can oscillate about a critical power $|t|^{mN/(N-2m)}$ and may not satisfy that $t \mapsto t^{-1}f(x,t)$ is increasing.
	
	\begin{example} In this work, we consider self-similar functions with power $\sigma = (N-2m)/m$ or $\sigma = N/m,$ and with factor $\gamma =2.$
		Typical examples are:
		
		\noindent$\mathrm{(i)}$ $F(t)=|t|^{mN/(N-2m)},$ with power $\sigma =(N-2m)/m$ any factor $\gamma>1$;
		
		\noindent$\mathrm{(ii)}$ $H(t)=\cos (\ln |t| ) |t|^{mN/(N-2m)},$ $H(0):=0,$ with factor $\gamma = e^{2m \pi /(N-2m)}$ and power $\sigma = (N-2m)/N.$
		
		\noindent$\mathrm{(iii)}$ $\hat{H}(t)=\exp\{ \sin (\Theta \ln |t| ) \}|t|^{mN/(N-2m)},$ $\hat{H}(0):=0,$ where $\Theta = 2 \pi m /(\ln2 (N-2m) ),$ is self-similar with $\gamma = 2$ and $\sigma = (N-2m)/N.$
	\end{example}

	\begin{example}\label{ex_ss_g}The pure power $|t|^{q-2}t$ with $q=2^\ast$ is the simplest example of self-similar homeomorphism, for the space $\mathcal{D}^{1,2}(\mathbb{R}^N).$ To have a nontrivial example take  $\gamma>1, \;  I_0 = (1,\gamma ^{N/m})$ and $h_0(t) \in C(\overline{I_0})\cap C^1(I_0) $ satisfying $h_0(\gamma ^{N/m} )=\gamma ^N h_0 (1),$ \eqref{g_dois}, \eqref{g_tres} and $h_0' (t)\geq 0$ for all $ t\in   I_0.$  Setting $I_j = [\gamma ^{\sigma j}, \gamma ^{\sigma (j+1)} ),$ $j \in \mathbb{Z}$, for each $t \in I_j,$ define $h_1(t) = \gamma ^{(N-N/m)j}h_0 (\gamma ^{-(N/m) j} t ),$ if $t>0,$ $h_1 (0) = 0 $ and $h_2(t) = -h_1(t),$ if $t<0.$ Then the function given by $h(t):=h_1(t),$ if $t \geq 0,$ and $h(t) = h_2(t)$ if $t<0,$ is a nondecreasing homeomorphism and taking $g(t) = h^{-1}(t),$ we see that $g(t)$ satisfies \eqref{g_dois}, \eqref{g_tres} and \eqref{g_quatro}, when $\gamma =2$ (see \cite[p. 116]{tintabook}).
	\end{example}

	\section{Regularity results and Pohozaev type identity}\label{s_regularity}
	
	\subsection{Regularity results}  First we have  pointwise estimate for $u \in \mathcal{D}^{1,\theta }_{\rad}(\mathbb{R}^N)$.

	\begin{lemma}\label{l_radial} Let $N> \theta >1$.  Every function $u \in \mathcal{D}^{1,\theta }_{\rad}(\mathbb{R}^N)$  is almost everywhere equal to a funtion $\overline{u} \in C(\mathbb{R}^N \setminus \{0\} )$ such that
		\begin{equation*} 
			|\overline{u}(x)|\leq C \| \nabla u \| _\theta |x|^{\frac{\theta-N }{\theta}} ,\ \forall \, x \in \mathbb{R}^N \setminus \{0\},
		\end{equation*}
		where  $C$ depends only on $N$ and $ \theta$. 
		Consequently, if  $u \in \mathcal{D}^{2,\theta }_{\rad}(\mathbb{R}^N),$ then $\overline{u} \in C^1(\mathbb{R}^N \setminus \{0\} ),$
		\begin{equation*}
			|\overline{u} (x)|\leq C \| \nabla u \| _{\frac{\theta N}{N-\theta }} |x|^{\frac{2\theta-N }{\theta}} \quad\text{and}\quad|\nabla \overline{u}  (x)| \leq C \| \Delta u \| _{\theta } |x| ^{\frac{\theta - N}{\theta }} ,\ \forall x \in \mathbb{R}^N \setminus \{0\}.
		\end{equation*}
		In particular, for a smooth bounded domain $ \Omega\subset \mathbb{R}^N $ such that $ \overline{\Omega } \cap \{0\}= \emptyset,$ there is $ \beta \in (N, \infty) $ such that
		$$
		|\overline{u}(x)-\overline{u}(y)|\leq C_\Omega|x-y|^{1-N/\beta}\|u\|_{W^{1,\beta }(\Omega ) } \quad \text{for all} \quad x,y \in \Omega.
		$$
	\end{lemma}
	\begin{proof}The proof follows by standard arguments once we take into account the continuous embedding $\mathcal{D}^{2,\theta}(\mathbb{R}^N) \hookrightarrow \mathcal{D}^{1,\theta N/(N-\theta)}(\mathbb{R}^N)$ and that for $\varphi \in \mathcal{D}^{2,\theta}(\mathbb{R}^N),$ we have $\varphi_{x_j} \in \mathcal{D}^{1,\theta}(\mathbb{R}^N),$ for $j=1,\ldots, N$ (see also \cite{strauss} for further details).
	\end{proof}
	\begin{proof}[Proof of Theorem \ref{p_regcrit}]
		If $\Omega$ is a bounded domain with smooth boundary such that $0 \not \in \overline{\Omega } ,$ then $v=h \left((b(|x|))^{-1} L_\lambda (u) \right) \in L^q (\mathbb{R}^N)$ is a very weak solution of the equation $-\Delta v = \hat{f} (x)$ in $\Omega,$ where $\hat{f} (x) = f(|x|,u) + \lambda |x|^{-2}v,$ that is, 
		$\int_{\Omega} v \Delta \varphi \dx  = \int_{\Omega} \hat{f} (x) \varphi \dx  $ for all $\varphi \in C_c^{\infty} (\Omega)  $.
		By Lemma \ref{l_radial}, we see that $\hat{f} (x) \in L^q(\Omega),$ and therefore we can apply \cite[Proposition 1.1]{veryweak}, to conclude that $v\in W^{2,q}_{\loca} (\Omega ),$ and $(u,v)$ is a radial strong solution for Syst.~\eqref{S}. Moreover,
		\begin{equation*}
			\int_{\mathbb{R}^N}|L_\lambda (v)|^l \dx= \int_{\mathbb{R}^N}|f(x,u)|^l \dx\leq C\int_{\mathbb{R}^N}|u|^p \dx,
		\end{equation*}
		and so, using Corollary \ref{c_charaD}, we see that $v\in \mathcal{D}^{2,l}_\rad (\mathbb{R}^N).$
	\end{proof}
	In the context of Theorem \ref{p_regcrit}, one can notice that the decay given by Lemma \ref{l_radial} cannot be applied at the origin. Therefore, when $\lambda =0$, we need a different approach from the one above and some additional assumptions. To study this case, we improve some arguments on regularity for hamiltonian systems contained in \cite{hulshof,vandervorst_best}.
	\begin{proof}[Proof of Theorem \ref{c_regularity_crit_0}]
		We first observe that by Theorem \ref{p_regcrit}, we only need to prove that $\limsup_{x \rightarrow 0}|u(x)|<\infty.$ Also, let us recall that here we have \eqref{HYPcrit}, and so $p,q > N/(N-2)$.
		
		\textit{Case $p,q > 2$: } The super linear case follows by using a similar argument as in \cite[Theorem 3.1]{hulshof}. Hence, if $N=3$ or $N=4$ we are done.
		
		\textit{Case $p\leq2$ or $q\leq 2$ ($N\geq 5 $):}  First, let us consider $q\leq 2,$ which is equilvent to $p\geq 2_\ast = 2N/(N-4)>2.$  For this case, we separate the proof in five steps.
		
		\noindent \textit{Step 1.} We prove that for each $\varepsilon >0$ there exists $q_\varepsilon \in L^{p/(p-2)}$ and $f_\varepsilon \in L^\infty (\mathbb{R}^N) \cap L^{p/(p-1)} (\mathbb{R}^N)$ such that $f(x,u) = a(x)u = q_\varepsilon u + f_\varepsilon,$ almost everywhere in $\mathbb{R}^N,$ where $a(x) = f(x,u)/u$ and $\| q_\varepsilon \|_{p/(p-2)} <\varepsilon.$
		
		To see this, following \cite[Lemma B2]{vandervorst_best} we explicitly define $q_\varepsilon$ and $f_\varepsilon$ for which the next properties can be checked:
		$$
		\begin{aligned}
			\mathcal{A}_k = \;  & \{ x \in \mathbb{R}^N : |a(x)| < k \}, \\
			\mathcal{U}_n = \; &  \{ x \in \mathbb{R}^N : |u(x)| < n \} ,\text{ where } k,n \, \in \mathbb{N},\\
			q_\varepsilon (x)= \; &
			\left\{
			\begin{aligned}
				&\frac{1}{n_\varepsilon}a(x), \quad  & \text{if} & \quad   x \in  \mathcal{A}_k \cap \mathcal{U}_n,  \\
				& a(x),\quad & \text{if} & \quad   x \in \mathbb{R}^N \setminus (\mathcal{A}_k \cap \mathcal{U}_n),
			\end{aligned}
			\right.
			\\
			f_\varepsilon (x) = \; & (a(x) - q _\varepsilon(x) )u(x).
		\end{aligned}
		$$
		Therefore, $f_\varepsilon = 0$ on $\mathbb{R}^N \setminus (\mathcal{A}_k \cap \mathcal{U}_n).$
		
		\noindent \textit{Step 2.} Denote $\mathcal{E}(w ) = (- \Delta ) \left(h ((b(x))^{-1} (-\Delta w))\right)$ the fourth-order operator given by \eqref{P}. We prove that there exists a well defined solution operator $\mathcal{S} = \mathcal{E}^{-1} : L^{l}(\mathbb{R}^N) \rightarrow \mathcal{D}^{2,m}(\mathbb{R}^N).$ Precisely,  given $\xi \in L^{l}(\mathbb{R}^N)$ there exists a unique $u_\xi $ ($=\mathcal{E}^{-1}(\xi )$) $\in \mathcal{D}^{2,m}(\mathbb{R}^N)$ such that $\mathcal{E}(u_\xi) = \xi,$ in the weak sense, that is,
		\begin{equation}\label{operator_weakform}
			\int _{\mathbb{R}^N } h \left((b(x))^{-1}  (-\Delta u_\xi)\right) (-\Delta \varphi )\dx  = \int_{\mathbb{R}^N} \xi \varphi \dx,\ \forall \, \varphi \in \mathcal{D}^{2,m}(\mathbb{R}^N).
		\end{equation}
		In order to do that, we consider the functional $L_{\xi} : \mathcal{D}^{2,m} (\mathbb{R}^N) \rightarrow \mathbb{R}$ given by 
		\begin{equation*}
			L_{\xi}(w) = \int _{\mathbb{R}^N }b(x)H ((b(x))^{-1}(-\Delta w) ) \dx - \int _{\mathbb{R}^N}\xi w \dx.
		\end{equation*}
		It is easy to see that critical points of $L_{\xi}$ satisfy \eqref{operator_weakform} and so they are the solutions of $\mathcal{E}(w) = \xi,$ in the weak sense. 
		The existence and uniqueness of a critical point of $L_\xi$ are ensured by the fact that $L_\xi$ is coercive and sequentially weakly lower semicontinuous.
		
		\noindent\textit{Step 3.} We prove that the operator $\mathcal{E}^{-1}$ can be characterized as compositions of Riesz potentials. More precisely, let $\mathcal{R}=\mathcal{R}_\theta : L ^\theta (\mathbb{R}^N)\rightarrow L ^\mu (\mathbb{R}^N) ,$ $1\leq \theta < \mu <\infty,$ given by
		\begin{equation*}
			\mathcal{R}_\theta[w](x) = c_N \int _{\mathbb{R}^N} w(y) |x-y|^{2-N} \dy,\ \text{where} \ \frac{1}{\mu} = \frac{1}{\theta} - \frac{2}{N}\text{ and }c_N=\frac{\Gamma\left( \frac{N}{2} - 2\right)}{4\pi^{\frac{N}{2}} \Gamma(1)}.
		\end{equation*}
		Then $\mathcal{E}^{-1} (\xi ) = \mathcal{R}_m \left[b(x) g \left( \mathcal{R}_l [\xi] \right) \right].$
		
		To prove this, we begin by recalling that $(-\Delta ) (\mathcal{R}[\varphi]) = \varphi $ for all $\varphi \in C^\infty _0 (\mathbb{R}^N),$ also the Hardy-Littlewood-Sobolev inequality:  $\| \mathcal{R}_\theta[w] \| _\mu \leq C \| w\|_\theta, $ for all $w \in L ^\theta (\mathbb{R}^N)$ (see more details in \cite{stein}). Now let $(\varphi_k ) \subset C^\infty _0 (\mathbb{R}^N)$ such that $\varphi_k \rightarrow \xi $ in $L^l (\mathbb{R}^N).$ We have that
		\begin{equation*}
			\int _{\mathbb{R}^N}|\Delta (\mathcal{R}_l [\varphi _k - \varphi _i])|^l \dx = \int_{\mathbb{R}^N} |\varphi _k - \varphi _i|^l\dx,\ \forall \, k,i \in \mathbb{N},
		\end{equation*}
		which means that $(\mathcal{R}_l [\varphi _k])$ is a Cauchy sequence in $\mathcal{D}^{2,l}(\mathbb{R}^N),$ and so it must converge for some $v_l \in \mathcal{D}^{2,l}(\mathbb{R}^N).$ On the other hand, by the Hardy-Littlewood-Sobolev inequality (with $\theta  = l$), we see that $\mathcal{R}_l [\varphi _k] \rightarrow \mathcal{R}_l [\xi]$ in $L^{q}(\mathbb{R}^N),$ and so $v_l = \mathcal{R}_l [\xi].$ Let $\varphi \in C^\infty _0(\mathbb{R}^N),$ we have, by Lebesgue's dominated convergence theorem,
		\begin{equation*}
			\int_{\mathbb{R}^N}\left(-\Delta(\mathcal{R}_l [\xi]) \right) \varphi \dx = \lim_{k \rightarrow \infty }\int_{\mathbb{R}^N}\left(-\Delta(\mathcal{R}_l [\varphi_k]) \right) \varphi \dx  = \int_{\mathbb{R}^N}\xi  \varphi \dx.
		\end{equation*}
		Consequently, $-\Delta(\mathcal{R}_l [\xi]) = \xi.$ By the same regularity result used in Theorem \ref{p_regcrit}, we can see that $v_\xi := h \left((b(x))^{-1}  (-\Delta u_\xi)\right) \in \mathcal{D}^{2,l}(\mathbb{R}^N)$ (see \eqref{operator_weakform}). Since $-\Delta v_\xi = \xi,$ we conclude that $v_\xi = \mathcal{R}_l [\xi].$ Setting $\hat{\xi}= b(x)g(v_\xi) \in L^m(\mathbb{R}^N)$ and using an analogous argument as above, we obtain that $u_\xi = \mathcal{R}_m[\hat{\xi}] = \mathcal{R}_m [b(x)g(\mathcal{R}_l [\xi])] .$
		
		\noindent \textit{Step 4.} Given $R>0,$ we prove that the operator $\mathcal{B}_\varepsilon(w) := \mathcal{E}^{-1}(q_\varepsilon w\mathcal{X}_{B_R})$ acts in $L^p(\mathbb{R}^N)$ on $L^p(\mathbb{R}^N).$ Moreover, $\|\mathcal{B}_\varepsilon(w) \|_p \leq C_R\varepsilon^{q-1}\|w\mathcal{X}_{B_R}\|^{q-1}_p,$ for all $w \in L^p(\mathbb{R}^N),$ and some $\hat{C}_R>0$ independent of $\varepsilon.$
		
		To prove this, we note that by H\"{o}lder inequality $q_\varepsilon w \in L^{l}(\mathbb{R}^N),$ provided that $w \in L^p(\mathbb{R}^N).$ Now in order to apply the Hardy-Littlewood-Sobolev inequality, we use the characterization given in \textit{Step 3}, 
		\begin{equation*}
			\|\mathcal{B}_\varepsilon(w)\|_p \leq C_1 \|\mathcal{R}[q_\varepsilon w\mathcal{X}_{B_R}] \|^{q-1}_{\mu _p (q-1) }\leq C_2\| q_\varepsilon w\mathcal{X}_{B_R} \|_{\mu_ {p,q}} ^{q-1},
		\end{equation*}
		where
		\begin{equation*}
			\frac{1}{p} = \frac{1}{\mu_p} - \frac{2}{N}\quad \text{and}\quad\frac{1}{\mu_p (q-1)} = \frac{1}{\mu_{p,q}} - \frac{2}{N}.
		\end{equation*}
		For $\sigma = p/(p-\mu_{p,q}),$ H\"{o}lder inequality implies that $\| q_\varepsilon w \mathcal{X}_{B_R}\|_{\mu_{p,q}}\leq \|q_\varepsilon \mathcal{X}_{B_R}\| _{\sigma\mu_{p,q}} \|w \mathcal{X}_{B_R} \|_p.$ Moreover, since $q\leq 2,$ we see that
		\begin{equation*}
			\sigma \mu_{p,q} = \frac{Nq(q-1)}{2q^2 - (q-2) [(q-1)N - 2q]} \leq \frac{qN}{4q-(q-2)N} = \frac{p}{p-2},
		\end{equation*}
		and the equality occurs when $q=2.$ Therefore, $\|q_\varepsilon \mathcal{X}_{B_R}\| _{\sigma\mu_{p,q}} \leq |B_R|^\zeta \|q_\varepsilon \mathcal{X}_{B_R} \|_{p/(p-2)},$ $\zeta=(p-2)/(p-\mu_{p,q}(p-2))$ (when $q<2$), and we conclude, by Step $1,$ that $\|\mathcal{B}_\varepsilon(w)\|_p \leq C_R\varepsilon^{q-1}\|w\mathcal{X}_{B_R}\|^{q-1}_p,$ where $C_R>0$ does not depend on $\varepsilon.$
		
		\noindent \textit{Step 5.} Now let us argue by contradiction and assume that $\lim_{x\rightarrow 0}|u(x)| = +\infty.$ In this case, using \eqref{f_porbaixo}, we can assume that in the \textit{Step 1}, $k,n$ are taken in a such way that $B_\delta \subset \mathbb{R}^N \setminus( \mathcal{A}_k \cap \mathcal{U}_n),$ for some $\delta >0.$ We have seen in Theorem \ref{p_regcrit} and \textit{Step 1} that $\mathcal{E}(u) = q_\varepsilon u + f_\varepsilon.$ Hence $\mathcal{E}^{-1} (\mathcal{X}_{B_{\delta}}\mathcal{E}(u) )=\mathcal{E}^{-1}( \mathcal{X}_{B_{\delta}}(q_\varepsilon u + f_\varepsilon ) ) =\mathcal{E}^{-1}( \mathcal{X}_{B_{\delta}} q_\varepsilon u ).$ We now observe that there exists a constant $\mathcal{C}_u>0$ (depending only on $u$) such that $\mathcal{C}_u \| u \mathcal{X}_{B_{\delta}} \|_p \leq   \| \mathcal{E}^{-1} (\mathcal{X}_{B_{\delta}}\mathcal{E}(u) ) \mathcal{X}_{B_{\delta}} \|_p.$ Hence, from \textit{Step 4},
		\begin{align*}
			0 &= \|\mathcal{E}^{-1} (\mathcal{X}_{B_{\delta}}\mathcal{E}(u) ) \mathcal{X}_{B_{\delta}}-   \mathcal{E}^{-1}( \mathcal{X}_{B_{\delta}} q_\varepsilon u )\mathcal{X}_{B_{\delta}}\|_p\\
			&\geq (\mathcal{C}_u  -  C_1 \varepsilon^{q-1} \| u \mathcal{X}_{B_{\delta}} \|_p ^{q-2}) \| u \mathcal{X}_{B_{\delta}} \|_p.
		\end{align*}
		Taking $\varepsilon $ small enough we see that $u=0$ in $B_{\delta},$ a contradiction with fact that $|u(x)| \rightarrow +\infty,$ as $|x|\rightarrow 0.$
		
		\textit{Case where $p\leq 2$:} In this case we have $q> 2. $ Thus our hyphotheses allow us to consider the reduction of Syst.~\eqref{S} to the equation $-\Delta( (\hat{b}(x))^{-1} f^{-1} (-\Delta v) ) = b(x)g(v),$ and use the same analysis made before for the case that $q\leq 2$ (interchange $u$ by $v$).
	\end{proof}
	\begin{remark}
		When $q=2,$ the proof of this result is much more simpler, since the $\mathcal{E}^{-1}$ is linear.
	\end{remark}
	For the case that nonlinearities have subcritical growth in the hyperbola sense, thanks to \cite[Proposition 1.1]{veryweak}, we can approach the local regularity of Eq. \eqref{P} in a similar way to what is usually made for the one equation case (using a bootstrap type argument).
	\begin{proof}[Proof of Theorem \ref{p_regsub}]Let $\Omega$ be a open bounded domain with smooth boundary. So $v=h \left( (b(x))^{-1}L (u) \right) \in L^q (\mathbb{R}^N)$ is a very weak solution of the equation $-\Delta v = \hat{f} (x)$ in $\Omega,$ where $\hat{f} (x) = f(u)-V(x)v\in L^q(\Omega).$ So we have that $v \in W_{\loca}^{2, q}(\mathbb{R}^N),$ which implies that $(u,v)$ is a strong solution for Syst.~\eqref{S}. If $N-2m \leq 0,$ then Sobolev embedding implies that $u \in L^\theta (\mathbb{R}^N)$ for all $\theta \in (m, \infty).$ Successively applications of \cite[Proposition 1.1]{veryweak} leads to $v \in W_{\loca}^{2, \theta}(\mathbb{R}^N),$ for all $\theta \in (1,\infty),$ and therefore $u,v \in C^1(\mathbb{R}^N).$ In this case if one assume that $f(x,t),$ $b(x)g(t)\in C^{0,\alpha }_{\loca } (\mathbb{R}^N\times \mathbb{R}),$ $V(x) \in C^{0,\alpha }_{\loca } (\mathbb{R}^N),$ then Schauder regularity theory can be applied to conclude that $u,v \in C^{2,\alpha}_{\loca} (\mathbb{R}^N).$ 
		
		If $N-2m > 0,$ then $f(u) \in L_{\loca }^{m_\ast / p_\varepsilon}(\mathbb{R}^N),$ and so $\hat{f}(x) \in L ^{m_1} _{\loca} (\mathbb{R}^N),$ where $m_1  = \min\{m_\ast / p_\varepsilon ,q \}.$ Therefore \cite[Proposition 1.1]{veryweak} implies that $v \in W^{2, m_1}_{\loca} (\mathbb{R}^N).$ In particular we get that $(u,v)$ is a strong solution for \eqref{S}. Now note that $f(u) \in L ^{p_\varepsilon / (p_\varepsilon -1 )} (\Omega ),$ which allows us to apply the standard $L^p$ regularity theory to get that $v \in  W^{2,p_\varepsilon / (p_\varepsilon -1 )}_{\loca } (\Omega ).$ Denoting $l^{(\varepsilon )} = p_\varepsilon / (p_\varepsilon -1 ),$ we see that $g(v) \in L^{l _\ast / (q-1)}_{\loca}(\Omega),$ and so $u \in W^{2,\sigma_1 } _{\loca} (\Omega ),$ where $\sigma_1 = l_\ast/(q-1).$ If $N-2\sigma_1 \leq0,$ we argue as above and conclude the result. If not, by Sobolev embedding, we see that $u \in L^{(\sigma_1)_\ast}_{\loca }(\Omega)$ and $f(u) \in L_{\loca}^{(\sigma_1)_\ast /(p_\varepsilon-1)}(\Omega).$ Similarly $v \in W^{2,(\sigma_1)_\ast /(p_\varepsilon-1)}_{\loca } (\Omega ),$ $g(v) \in L^{\sigma_2}_{\loca}(\Omega)$ and $u \in W^{2,\sigma_2 } _{\loca} (\Omega ),$ where $\sigma_2 = (\sigma_1)_\ast /  [(p_\varepsilon - 1)(q-1)].$ Once again, if $N-2\sigma_2 \leq 0,$ then the proof is complete. Hence we suppose that $N-2\sigma_2 > 0.$ Proceeding with this argument we find that $\sigma_k = (\sigma_{k-1})_\ast  /  [(p_\varepsilon - 1)(q-1)],$ $k \geq 2,$ where $N-2\sigma _{k-1} >0.$ Using the definition of $(\sigma _k)_\ast$ we see that
		\begin{align}
			\label{bootstrap} \frac{1}{ (\sigma_k )_\ast } &= \frac{1}{(\sigma_1)_\ast}[(p_\varepsilon - 1)(q - 1)]^{k-1} - \frac{2}{N} \sum _{j=0}^{k-2} [(p-1)(q-1)]^j
			\\ &= \left[\frac{1}{ (\sigma_1)_\ast } - \frac{2}{N}\frac{1}{(p_\varepsilon-1)(q-1) -1} \right] [(p_\varepsilon - 1)(q - 1)]^{k-1}\nonumber
			\\ & \qquad  +\frac{2}{N}\frac{1}{(p_\varepsilon - 1)(q-1) - 1}.\nonumber
		\end{align}
		The numerator of the coefficient for the first term in the right-hand side of \eqref{bootstrap} equals to
		\begin{equation*}
			\left[ (p-1)(q-1) -1 \right] \left[N (p-1)(q-1) - 2pq\right]  -2p,
		\end{equation*}
		and it is negative. Therefore, for $k$ large enough we have that $1/(\sigma_k )_\ast <0$ and consequently $W^{2,\sigma _k}_{\loca }(\Omega) \hookrightarrow L^\theta_{\loca} (\mathbb{R}^N),$  for all $\theta \in [1,\infty ).$ For the case that $\eqref{subDACERTO}$ holds, we see that $f(u)\in L^{p/(p-1)}(\mathbb{R}^N)$ and since $V(x)\in L ^\infty (\mathbb{R}^N)$ we can apply the $L^p$ elliptic regularity theory to obtain that $v \in W^{2,l}(\mathbb{R}^N) = W^{2,l}_V(\mathbb{R}^N).$
	\end{proof}
	\subsection{Pohozaev type identity} 
	In this section, we prove Theorem \ref{p_pohozaev_id}, and we study its consequences.
	\begin{proof}[Proof of Theorem \ref{p_pohozaev_id}]
		Consider $\xi \in C_0 ^\infty (\mathbb{R}:[0,1])$ such that $\xi(t)= 1,$ if $|t| \leq 1,$ $\xi(t)= 0,$ if $|t| \geq 2,$ and $|\xi ' (t)| \leq C,$ $\forall \, t \in \mathbb{R}$ for $C>0.$ Let $\mathcal{O}=\{x^{(1)},\ldots,x^{(l)} \}.$ For each $n=1,\ldots,$ define $\xi _n : \mathbb{R}^{N} \rightarrow \mathbb{R}$ by
		\begin{align*}
			\xi_n (x)=
			\left\{
			\begin{aligned}
				&\xi(|x|^2/n^2), \quad&\text{if }|x-x^{(i)}|^2 > 2/n^2,&\ \text{ for all }i=1,\ldots,l,\\
				&1-\xi (n^2 |x-x^{(i)}|^2),\quad&\text{if }  |x-x^{(i)}|^2 \leq 2/n^2,&\ \text{ for some }i=1,\ldots,l.
			\end{aligned}
			\right.
		\end{align*}
		Then, for $n$ large enough, $\xi _n$ is well defined ($\xi_n \equiv 1$ in $B_{n^2}$), $\xi _n \in C_0 ^\infty (\mathbb{R}^N)$ and verifies $|x||\nabla \xi _n (x)| \leq C,$ $\forall \, x\in \mathbb{R}^{N},$ for some $C>0.$ Some computations leads to the following identities
		\begin{equation}\label{poho1}
			\begin{aligned}
				\xi _n \left\langle x, \nabla v  \right\rangle \Delta u &= \dive \left(\xi _n \left\langle x, \nabla u  \right\rangle \nabla v- \xi _n  \left\langle \nabla u, \nabla v \right\rangle x \right) \\
				& \quad - \left\langle x, \nabla v\right\rangle \left\langle\nabla \xi _n , \nabla u \right\rangle + \left\langle x, \nabla \xi _n \right\rangle \left\langle\nabla u, \nabla v \right\rangle \\
				& \quad + \xi _n N  \left\langle \nabla u , \nabla v \right\rangle - \xi _n \left\langle \nabla u, \nabla v \right\rangle + \xi _n \sum_{j=1}^N \sum _{i=1}^N x_j u_{x_i x_j} v_{x_i},
			\end{aligned}
		\end{equation}
		\begin{equation*}
			\begin{aligned}
				\xi _n \left\langle x, \nabla u  \right\rangle \Delta v &= \dive \left(\xi _n \left\langle x, \nabla v  \right\rangle \nabla u- \xi _n  \left\langle \nabla v, \nabla u \right\rangle x \right)  \\
				&\quad - \left\langle x, \nabla u\right\rangle \left\langle\nabla \xi _n , \nabla v \right\rangle + \left\langle x, \nabla \xi _n \right\rangle \left\langle\nabla v, \nabla u \right\rangle \\
				& \quad + \xi _n N  \left\langle \nabla v , \nabla u \right\rangle - \xi _n \left\langle \nabla v, \nabla u \right\rangle + \xi _n \sum_{j=1}^N \sum _{i=1}^N x_j v_{x_i x_j} u_{x_i},
			\end{aligned}
		\end{equation*}
		Moreover, denoting $\Psi= F(x,u) + b(x)G(v) - V(x)uv,$ we have
		\begin{multline}
			(f(x,u) - V(x)v) \left\langle x, \nabla u \right\rangle \xi _n + (b(x)g(v) - V(x)u)\left\langle x, \nabla v\right\rangle \xi _n \\ = \dive (\xi \Psi x) - \Psi \left\langle x, \nabla \xi _n \right\rangle -\xi _n N  \Psi  - \xi _n \sum _{i=1}^N x_i F_{x_i}(x,u) \\ - \xi _n G(v) \left\langle x, \nabla b(x) \right\rangle + \xi _n \left\langle x, \nabla V(x) \right\rangle uv.\label{poho3}
		\end{multline}
		We also have that
		\begin{equation*}
			\sum_{j=1}^N \sum _{i=1}^N \int_{\mathbb{R}^N}\xi n x_j  \left(v{x_i x_j} u_{x_i}+v_{x_i x_j} u_{x_i}\right) \dx=-\int_{\mathbb{R}^N}(\left\langle \nabla \xi _n ,  x\right\rangle +N \xi _n) \left\langle \nabla u, \nabla v\right\rangle  \dx.
		\end{equation*}
		Multiplying  the first equation of \eqref{S} by $\xi _n \left\langle x, \nabla v  \right\rangle$ and the second one by $\xi _n \left\langle x, \nabla u  \right\rangle,$ suming  the resulting identities, integrating over $\mathbb{R}^N,$ and applying  \eqref{poho1}--\eqref{poho3}, we obtain 
		\begin{equation*}
			(2-N) \int_{\mathbb{R}^N}\left\langle\nabla u , \nabla v \right\rangle \dx = \int_{\mathbb{R}^N} \left\langle x, \nabla V(x) \right\rangle uv -(N\Psi + \sum _{i=1}^N x_i F{x_i}(x,u) + G(v) \left\langle x, \nabla b(x) \right\rangle)\dx,
		\end{equation*}
		where we have used Lebesgue's dominated convergence theorem to pass the limit as $n \rightarrow \infty.$ Using condition \eqref{condicao_poho}, identity \eqref{SIS_poho} follows.
		
		To prove \eqref{EQ_poho} it suffices to use the fact that $G(h (t)) = t h(t) - H(t),$ $\forall \, t\in \mathbb{R},$ and
		\begin{equation*}
			\int_{\mathbb{R}^N} f(x,u)u \dx = \int_{\mathbb{R}^N}h \left((b(x))^{-1} L(u)\right) L(u) \dx  = \int_{\mathbb{R}^N}v b(x) g(v)  \dx\qedhere
		\end{equation*}
	\end{proof}
	\begin{corollary}\label{cor_pohozaev} 
		Assume \eqref{constant_ss}, \eqref{g_dois}, $b(x)\equiv 1$ and that $V(x) \equiv  -\lambda |x|^{-2}$ with $0 \leq \lambda < \Lambda _{N,m}^{1/m}$.  Let $u\in \mathcal{D}^{2,m} (\mathbb{R}^N)$ be a weak solution of \eqref{P}. Then
		\begin{equation*}
			\int_{\mathbb{R}^N}H(L_\lambda (u))- \frac{2}{N}h(L_\lambda (u))L_\lambda (u)\dx = \int _{\mathbb{R}^N}F(u) \dx.
		\end{equation*}
		In particular,
		\begin{equation*}
			\int_{\mathbb{R}^N}|L_\lambda (u)|^m\dx = \frac{mN}{N-2m}\int _{\mathbb{R}^N}F(u) \dx = p \int _{\mathbb{R}^N}F(u) \dx.
		\end{equation*}
	\end{corollary}
	As a consequence of Theorem \ref{p_pohozaev_id}, we now have a general nonexistence result.
	\begin{proposition}\label{p_nonexist} 
		Suppose that $V(x)\in C^1(\mathbb{R}^N \setminus \mathcal{O}),$ 
		where $\mathcal{O}$ is a finite set, $0<\beta \leq b(x)\in L^\infty (\mathbb{R}^N)\cap C^1 (\mathbb{R}^N)$ 
		and that $(u,v) \in W_{\loca}^{2, q/(q-1)}(\mathbb{R}^N)\times W_{\loca}^{2, p/(p-1)}(\mathbb{R}^N )$ is a strong solution of \eqref{S} satisfying \eqref{condicao_poho_weak} and \eqref{condicao_poho}. 
		Furthermore suppose that one of the following assumptions is satisfied:
		
		\noindent$\mathrm{(i)}$ $f(x,t) \equiv b(x)f_0(t),$ $pF_0(t) \geq f_0(t)t\geq C|t|^p,$ $q G(t) \geq g(t)t\geq C|t|^q,$ $0<p,q < 2^\ast,$ $V(x)\equiv 0,$ and
		\begin{equation*}
			\frac{1}{\theta} \left(N b(x) + \left\langle x, \nabla b(x)\right\rangle \right) -  \frac{N-2}{2}b(x)>0,\, \text{ in }\mathbb{R}^N,\text{ for }\theta = p,q.
		\end{equation*}
		Next we assume that $b(x)\equiv 1,$ \eqref{g_dois} and suppose that $v= h \left(L (u) \right) \in W_{\loca}^{2, p/(p-1)}(\mathbb{R}^N).$
		
		\noindent$\mathrm{(ii)}$ $f(x,t) \equiv f(t),$ $f(t)t\geq p F(t),$ $2 V(x) + \left\langle x, \nabla V(x)\right\rangle \geq 0,$ with strict inequality in a nonzero measure domain, $u,v\geq 0$ almost everywhere in $\mathbb{R}^N$ and 
		\begin{equation}\label{assump_nonexist2}
			\frac{\mathcal{C}_2}{\mathcal{C}_1} \leq \left(\frac{p}{2p+N} \right) \frac{N}{m}.
		\end{equation}
		$\mathrm{(iii)}$ $f(x,t) \equiv f(t),$ $pF(t) \geq f(t)t,$ $V(x)\equiv 0$ and
		\begin{equation}\label{assump_nonexist}
			\frac{ \mathcal{C}_2}{\mathcal{C}_1} < \left(\frac{2p+N}{p} \right) \frac{m}{N}.
		\end{equation}
		Then $u=v=0.$ 
		
		\noindent$\mathrm{(iv)}$ Alternatively, if $f(x,t) \equiv b(x)f_0(t),$ $0\leq pF_0(t) \leq f_0(t)t,$ $0\leq q G(t) \leq g(t)t,$ $p,q \geq 2^\ast,$ $2 V(x) + \left\langle x, \nabla V(x)\right\rangle > 0$ in $\mathbb{R}^N,$
		\begin{equation*}
			\int_{\mathbb{R}^N}(2 V(x) + \left\langle x, \nabla V(x) \right\rangle )uv \dx \geq 0,
		\end{equation*}
		and
		\begin{equation*}
			0 \leq \frac{1}{\theta} \left(N b(x) + \left\langle x, \nabla b(x)\right\rangle \right) \leq \frac{N-2}{2}b(x),\, \text{ in }\mathbb{R}^N, \text{ for }\theta = p, q,
		\end{equation*}
		then $uv=0$ almost everywhere in $\mathbb{R}^N.$ Additionally, if $f_0(t)$ is a homeomorphism and $u,v\geq 0$ almost everywhere in $\mathbb{R}^N,$ then $u=v=0.$
	\end{proposition}
	
	\begin{proof}
		(i) We apply \eqref{SIS_poho}, $pF_0(t) \geq f_0(t)t$ and $q G(t) \geq g(t)t,$ to get
		\begin{multline*}
			\int _{\mathbb{R}^N } \left[ \left(\frac{N}{p} - \frac{N-2}{2} \right)b(x) + \frac{1}{p} \left\langle x, \nabla b(x)\right\rangle  \right]u f_0(u)  \dx \\ +  \int _{\mathbb{R}^N }\left[ \left(\frac{N}{q} - \frac{N-2}{2} \right)b(x) + \frac{1}{q} \left\langle x, \nabla b(x)\right\rangle  \right] v g(v)\dx \leq 0,
		\end{multline*}
		and so $u=v=0.$
		
		(ii) We argue by contradiction and assume that $u \neq 0.$ Using \eqref{EQ_poho}, we get
		\begin{equation*}
			N \int _{\mathbb{R}^N}H(L(u)) \dx + \int _{\mathbb{R}^N}h(L(u))\left[2 \Delta u + \left\langle x, \nabla V(x) \right\rangle u  -\frac{N}{p}L(u)\right] \dx \leq 0.
		\end{equation*}
		Now since $u,v \geq 0,$ we see that
		\begin{equation*}
			2 \Delta u + \left\langle x, \nabla V(x) \right\rangle u  -\frac{N}{p}L(u) \geq -\left(\frac{2p+N}{p}\right)L(u)\text{ a.e. in }\mathbb{R}^N,
		\end{equation*}
		which implies 
		\begin{equation*}
			\left[N - \left(2+\frac{N}{p} \right) \frac{\mathcal{C}_2}{\mathcal{C}_1}m   \right] \int _{\mathbb{R}^N}H(L(u)) \dx < 0.
		\end{equation*}
		This contradicts  \eqref{assump_nonexist2}.
		
		(iii) Applying \eqref{EQ_poho}, we have
		\begin{equation*}
			N \int _{\mathbb{R}^N} H(- \Delta u) \dx + \int _{\mathbb{R}^N}h(- \Delta u)(2\Delta u) \dx \geq \frac{N}{p}\int _{\mathbb{R}^N} f(u)u  \dx = \frac{N}{p} \int _{\mathbb{R}^N} h(- \Delta u) (- \Delta u) \dx.
		\end{equation*}
		Thus
		\begin{equation*}
			\left[\left( 2+\frac{N}{p} \right)\frac{\mathcal{C}_1}{\mathcal{C}_2}m - N \right] \int _{\mathbb{R}^N} H(- \Delta u) \dx \leq 0.
		\end{equation*}
		By \eqref{assump_nonexist}, we conclude $u=v=0.$
		
		(iv) Applying \eqref{SIS_poho}, and using $pF_0(t) \leq f_0(t)t,$ $q G(t) \leq g(t)t,$ we see that
		\begin{multline*}
			\int_{\mathbb{R}^N}(2V(x) + 2 \left\langle x, \nabla V(x) \right\rangle ) u v \dx \\ \leq  \int _{\mathbb{R}^N } \left[ \left(\frac{N}{p} - \frac{N-2}{2} \right)b(x) + \frac{1}{p} \left\langle x, \nabla b(x)\right\rangle  \right]u f_0(u)  \dx \\ +  \int _{\mathbb{R}^N }\left[ \left(\frac{N}{q} - \frac{N-2}{2} \right)b(x) + \frac{1}{q} \left\langle x, \nabla b(x)\right\rangle  \right] v g(v)\dx,
		\end{multline*}
		which implies $uv=0$ almost everywhere in $\mathbb{R}^N.$ Considering the additional hypothesis, let us define $A=\{x \in \mathbb{R}^N:u(x) = 0  \}$ and $B=\{ x \in \mathbb{R}^N: v(x) = 0\}.$ Clearly, $A\cup B = \mathbb{R}^N$ and by the fact that $g(t)$ is a homeomorphism, if $A \neq \emptyset,$ then $B \subset A = \mathbb{R}^N.$ Hence $L(v) = 0$ in $\mathbb{R}^N,$ which by the strong maximum principle implies $B= \mathbb{R}^N.$ When $B\neq \emptyset,$ since $f_0(t)$ is a homeomorphism, the same argument can be used to conclude $A=B=\mathbb{R}^N.$
	\end{proof}
	\begin{remark}
		Suppose that $\mathcal{C}_1 = \mathcal{C}_2.$ Then \eqref{assump_nonexist2} and \eqref{assump_nonexist} are equivalent to ask $1/p+1/q \leq (N-2)/N$ and $1/p+1/q \geq (N-2)/N$, respectively. Thus Corollary \ref{c_leconject} follows by Proposition \ref{p_nonexist}.
	\end{remark}
	
	\section{Concentration-compactness principle}\label{s_cc}
	In the following results, we have a concentration-compactness principle utilizing profile decomposition for weak convergence in $\mathcal{D}^{2,m}(\mathbb{R}^N)$ and $W^{2,m}(\mathbb{R}^N).$ 
	This kind of results were extensively studied in the past years (see for instance \cite{tintageral,tintabanach,paper1,paper2,wavelet,tintabook}).
	\begin{theorem}\label{teo_profcrit}
		Let $(u_k) \subset \mathcal{D}^{2,m}(\mathbb{R}^N)$ be a bounded sequence and  $N>2m.$ Then there exists $\mathbb{N}_{\ast } \subset \mathbb{N},$ disjoint sets $\mathbb{N} _{0}, \mathbb{N} _{- }, \mathbb{N}_ {+} \subset \mathbb{N}, $ with $\mathbb{N}_{\ast } = \mathbb{N} _{0} \cup  \mathbb{N}_ {+} \cup \mathbb{N} _{-} $ and sequences $(w ^{(n)}) _{n \in \mathbb{N}_{\ast } } \subset \mathcal{D}^{2,m}(\mathbb{R}^N),$ $(y_k ^{(n)}) _ {k \in \mathbb{N}} \subset \mathbb{Z}^N,$ $(j _k ^{(n)}) _{k \in \mathbb{N}}  \subset \mathbb{Z}$ for $n \in \mathbb{N}_{\ast }$ 
		such that, up to a subsequence,
		\begin{align}
			&2 ^{-\frac{N-2m}{m}j_k ^{(n)}}u_k (2 ^{- j_k ^{(n)}} \cdot + y_k^{(n)} )\rightharpoonup w^{(n)},&  k \rightarrow \infty,&& &\mbox{ in } \mathcal{D}^{2,m}(\mathbb{R}^N),\label{seis.um}\\ 
			&|j_k ^{(n)} - j_k ^{(i)}|+|2 ^ {j_k ^{(n)}}( y^{(n)}_k - y^{(i)}_k )| \rightarrow \infty ,&   k \rightarrow \infty,&& &\text{ for } i \neq n,\label{seis.dois} \\ 
			&\sum _{n \in \mathbb{N}_{\ast}}\| w^{(n)} \|_p^p  \leq \limsup_k \|u_k \|_p ^p,&\label{seis.tres} \\ 
			& u_k - \sum _{n \in \mathbb{N}_{\ast }} 2 ^{\frac{N-2m}{m} j_k ^{(n)} } w^{(n)}(2^{j _k ^{(n)}} ( \cdot - y_k^{(n)} ) ) \rightarrow  0, &k \rightarrow \infty,&& &\text{ in } L^{m_\ast }(\mathbb{R}^{N}),\label{seis.quatro}
		\end{align}
		and the series in \eqref{seis.quatro} converges uniformly in $k.$ Furthermore, $1 \in \mathbb{N} _0,$ $y_k ^{(1)} = 0;$ $j _k ^{(n)} = 0 $ whenever $n \in \mathbb{N}_0;$ $j_k ^{(n)} \rightarrow -\infty$ whenever $n \in \mathbb{N}_{- };$ and $j_k ^{(n)} \rightarrow +\infty$ whenever $n \in \mathbb{N}_{ +}.$
	\end{theorem}
	
	\begin{theorem}\label{teo_profsub}
		Let $(u_k) \subset W^{2,m}(\mathbb{R}^N)$ be a bounded sequence and $N \geq 2m.$ Then there exists $\mathbb{N}_{0 } \subset \mathbb{N},$ and sequences $(w ^{(n)}) _{n \in \mathbb{N}_{0 } } \subset W^{2,m}(\mathbb{R}^N),$ $(y_k ^{(n)}) _ {k \in \mathbb{N}} \subset \mathbb{Z}^N$ for $n \in \mathbb{N}_{0 }$ such that, up to a subsequence,
		\begin{align}
			&u_k (\cdot + y_k^{(n)} )\rightharpoonup w^{(n)},& \text{ as } k \rightarrow \infty,&& &\mbox{ in } W^{2,m}(\mathbb{R}^N),&\label{aplions}\\
			&|y^{(n)}_k - y^{(i)}_k| \rightarrow \infty , & \text{ as } k \rightarrow \infty,&& &\text{ for } i \neq n,\nonumber \\
			&\sum _{n \in \mathbb{N}_0}\| w^{(n)} \|_{m}^{m} \leq \limsup_k \|u_k \|_{m} ^{m},&\label{seis.tres.sub} \\
			&u_k - \sum _{n \in \mathbb{N}_0} w^{(n)}(\cdot - y_k^{(n)} ) \rightarrow  0,& \text{ as } k \rightarrow \infty,&& & \text{ in } L^{\theta}(\mathbb{R}^{N}),\label{seis.quatro.sub}
		\end{align}
		for any $\theta \in (m,m _\ast )$. Moreover, the series in \eqref{seis.quatro.sub} converges uniformly in $k.$
	\end{theorem}
	\begin{remark} $\mathrm{(i)}$ Theorem \ref{teo_profsub} can be seen as a generalization of the P.L. Lions compactness Lemma  (see \cite[Lemma I.1 and Remark I.4]{lionscompcase2}) and is a useful result when we are proving precompactness of Palais-Smale sequences for energy functionals associated with nonlinear elliptic problems. Indeed, taking  $w^{(n)} = 0,$ for all $n \in \mathbb{N}_0,$ we obtain that $u_k \rightarrow 0$ in $L^{\theta}(\mathbb{R}^{N})$ for all $\theta \in (m,m _\ast ).$ In this case, we argue by contradiction and, as a consequence, guarantee the existence of a non zero weak limit of a dislocated sequence as given in \eqref{aplions}.
		
		\noindent$\mathrm{(ii)}$ A natural question in this context is whether we can replace \eqref{seis.tres} in Theorem \ref{teo_profcrit} with
		\begin{equation}\label{aoutra}
			\sum _{n \in \mathbb{N}_{\ast}}\| \Delta w^{(n)} \|_m^m  \leq \limsup_k \|\Delta u_k \|_m ^m.
		\end{equation}
		Estimate \eqref{seis.tres} deals with the norm of the  Lebesgue space $L^p(\mathbb{R}^N),$ while \eqref{aoutra} with the norm of Sobolev space $\mathcal{D}^{2,m}(\mathbb{R}^N).$ Similarly, we can ask if we can substitute  \eqref{seis.tres.sub} in Theorem \ref{teo_profsub} with
		\begin{equation*}
			\sum _{n \in \mathbb{N}_0}\| w^{(n)} \|_{2,m}^{m } \leq \limsup_k \|u_k \|_{2,m} ^{ m },
		\end{equation*}
		where $\| \cdot \| _{2,m}$ stands for the usual Sobolev norm of $W^{2,m}(\mathbb{R}^N).$  If one can provide a positive answer for this open problems, then we can consider more general hypotheses to treat the existence of weak solutions for \eqref{P}  under our approach. See \cite{tintageral,tintabanach} for more detailed discussion on this subject.    
	\end{remark}
	We now describe the additional compactness given by the radial subspace of $\mathcal{D}^{2,m}(\mathbb{R}^N)$ in terms of the weak profile decomposition.
	\begin{proposition}\label{prop_radial}
		Let $(u_k)$ be a bounded sequence in $\mathcal{D}^{2,m}_\rad (\mathbb{R}^N),$ $(w^{(n)})$ and $(y_k^{(n)})$ be the collection of profiles given in Theorem \ref{teo_profcrit}. Then $(y_k^{(n)})_k = 0$ for all $n \in \mathbb{N}_\ast \setminus \{1\},$ $w^{(n)}\in \mathcal{D}^{2,m}_\rad (\mathbb{R}^N)$ and $\mathbb{N}_0 = \{1\}.$
	\end{proposition}
	\begin{proof}
		The idea is to find a new collection of profiles $(\overline{w}^{(n)}) _{n \in \mathbb{N} _{\sharp}},$ $(\overline{y}_k ^{(n)}) _ {k \in \mathbb{N}} \subset \mathbb{Z}^N$ and $(\overline{j} _k ^{(n)}) _{k \in \mathbb{N}}  \subset \mathbb{Z},$ $n \in \mathbb{N}_{\sharp},$ that satisfy the desired conditions. From this point, without loss of generality, we replace $w^{(n)}$ by $\overline{w}^{(n)}$ and $\mathbb{N}_\ast$ by $\mathbb{N}_{\sharp}$ in Theorem \ref{teo_profcrit}. In order to do this, we consider the set
		\begin{equation*}
			\mathbb{N} _{\sharp} = \left\lbrace n \in \mathbb{N} _\ast\setminus\{1\}:|2 ^{j_k ^{(n)}}y_k ^{(n)}|\text{ is bounded} \right\rbrace .
		\end{equation*}
		For each  $n \in \mathbb{N} _{\sharp}$, up to a subsequence,  standard diagonal argument yields,
		\[
		2 ^{j_k ^{(n)}}y_k ^{(n)} \rightarrow a^{(n)} \text{ as } k\rightarrow \infty.
		\] 
		Moreover, for $n\in \mathbb{N} _{\sharp},$ we have  
		\begin{equation*}
			2 ^{-\frac{N-2m}{m}j_k ^{(n)}} u_k (2 ^{-j_k ^{(n)}}\cdot ) -  2 ^{-\frac{N-2m}{m}j_k ^{(n)}} u_k (2 ^{-j_k ^{(n)}}(\cdot - a^{(n)})+y_k ^{(n)} ) \rightharpoonup 0,
		\end{equation*}
		in $\mathcal{D}^{2,m}(\mathbb{R}^N)$ as $k\rightarrow \infty.$
		Since $u \mapsto u(\cdot - a^{(n)})$ is linear and continuous,
		\begin{equation*}
			2 ^{-\frac{N-2m}{m}j_k ^{(n)}} u_k (2 ^{-j_k ^{(n)}}(\cdot - a^{(n)})+y_k ^{(n)} ) \rightharpoonup w^{(n)}(\cdot - a^{(n)}),
		\end{equation*}
		in $\mathcal{D}^{2,m}(\mathbb{R}^N),$ as $k\rightarrow \infty.$ Therefore
		\begin{equation*}
			2 ^{-\frac{N-2m}{m}j_k ^{(n)}} u_k (2 ^{-j_k ^{(n)}}\cdot ) \rightharpoonup w^{(n)}(\cdot - a^{(n)}) \text{ in }\mathcal{D}^{2,m}(\mathbb{R}^N),\text{ as }k\rightarrow \infty.
		\end{equation*}
		Based on above, we set 
		\begin{equation*}
			\left\{
			\begin{aligned}
				&\overline{w}^{(n)} = w^{(n)}(\cdot - a^{(n)}),\quad &&\overline{y}^{(n)}_k = 0,&\quad &\overline{j}^{(n)}_k = j_k ^{(n)},\quad&\text{for }n \in \mathbb{N} _{\sharp},\\
				&\overline{w}^{(n)} = w^{(n)},\quad&&\overline{y}^{(n)}_k = y^{(n)}_k,&\quad &\overline{j}^{(n)}_k = j_k ^{(n)},\quad&\text{for }n \in \mathbb{N}_\ast \setminus \mathbb{N} _{\sharp}.
			\end{aligned}
			\right.
		\end{equation*}
		It is easy to see that $(\overline{w}^{(n)})_{n \in \mathbb{N} _{\sharp}},$ $(\overline{y}_k ^{(n)}) _ {k \in \mathbb{N}}$ and $(\overline{j} _k ^{(n)}) _{k \in \mathbb{N}}$ satisfy conditions \eqref{seis.um}--\eqref{seis.tres}. To see that $(\overline{w}^{(n)})_{n \in \mathbb{N} _{\sharp}}$ also satisfies \eqref{seis.quatro}, we consider the  estimate
		\begin{equation}\label{CH1_gigante}
			\left\| u_k  - \sum _{n \in \mathbb{N}_{\ast }} \overline{d}^{(n)}_k \overline{w}^{(n)}  \right\| _{m_\ast}  \leq 
			\left\| u_k - \sum _{n \in \mathbb{N}_{\ast }} d^{(n)}_kw^{(n)} \right\|  _{m_\ast}  + \left\| \sum _{n \in \mathbb{N}_{\ast }} d^{(n)}_kw^{(n)} - \sum _{n \in \mathbb{N}_{\ast }} \overline{d}^{(n)}_k \overline{w}^{(n)} \right\|  _{m_\ast},
		\end{equation}
		where $d^{(n)}_k u = 2 ^{\frac{N-2m}{m} j_k ^{(n)} } u\big(2^{j _k ^{(n)}} ( \cdot - y_k^{(n)} )$ and $\overline{d}^{(n)}_k u = 2 ^{\frac{N-2m}{m} \overline{j}_k ^{(n)} } u\big(2^{\overline{j} _k ^{(n)}} ( \cdot - \overline{y}_k^{(n)} ),$ $u \in \mathcal{D}^{2,m}(\mathbb{R}^N).$ The first term in the right-hand side of inequality \eqref{CH1_gigante} goes to zero due to \eqref{seis.quatro}.
		Using diagonal argument, extracting successive subsequences for each $n \in \mathbb{N}_{\ast }$, $\sum _{n \in \mathbb{N}_{\ast }} \overline{d}^{(n)}_k \overline{w}^{(n)}$ is uniformly convergent in $k.$ 
		Thus from the uniform convergence of \eqref{seis.quatro}, we can assume that $\mathbb{N}_\ast$ is finite. Since
		\begin{align*}
			\left\| \sum _{n \in \mathbb{N}_{\ast }} d^{(n)}_kw^{(n)} - \sum _{n \in \mathbb{N}_{\ast }} \overline{d}^{(n)}_k \overline{w}^{(n)} \right\|  _{m_\ast} &\leq \sum _{n \in \mathbb{N} _{\sharp}} \left\| \delta_{j_k^{(n)}}( w^{(n)} - g_{a^{(n)} -2^{j_k^{(n)}} y_k ^{(n)}}  w^{(n)} )\right\|  _{m_\ast}\\
			&= \sum _{n \in \mathbb{N} _{\sharp}}\left\|  w^{(n)} - g_{a^{(n)} -2^{j_k^{(n)}} y_k ^{(n)}}  w^{(n)}  \right\|  _{m_\ast},
		\end{align*}
		where $\delta _j u(x) = 2 ^{\frac{N-2m}{m} j }u(2 ^j x), \; g_yu(x) = u(x-y),$ $u \in \mathcal{D}^{2,m}(\mathbb{R}^N),$ and the convergence to zero for the second term in \eqref{CH1_gigante} follows by Lemma \ref{l_convaux}.
		
		Now we prove: $\overline{w}^{(n)}\in \mathcal{D}^{2,m}_\rad (\mathbb{R}^N),$ $n \in \mathbb{N} _{\sharp}.$ Let $\eta \in  O(N)$ be an isometry of $\mathbb{R}^N.$ For $n \in \mathbb{N} _{\sharp},$ 
		\begin{equation}\label{CH1_weaklimit}
			\begin{aligned}2 ^{-\frac{N-2m}{m}\overline{j}_k ^{(n)}} u_k (2 ^{-\overline{j}_k ^{(n)}} \eta (x)) &= 2 ^{-\frac{N-2m}{m}\overline{j}_k ^{(n)}} ( u_k \circ \eta )(2 ^{-\overline{j}_k ^{(n)}} x) \\
				&= 2 ^{-\frac{N-2m}{m}\overline{j}_k ^{(n)}} u_k (2 ^{-\overline{j}_k ^{(n)}} x).
			\end{aligned}
		\end{equation}
		Since $T_\eta (u) = u \circ \eta$ is continuous in $\mathcal{D}^{2,m}(\mathbb{R}^N),$ passing  to the weak limit in \eqref{CH1_weaklimit}, we obtain $\overline{w}^{(n)} \circ \eta = \overline{w}^{(n)},$ for all $ n \in \mathbb{N} _{\sharp}$ and $\eta \in O(N).$  To conclude the proof we now show that $\overline{w}^{(n)} = 0$ for all $n \in \mathbb{N}_\ast \setminus \mathbb{N} _{\sharp}.$ 
		
		Arguing by contradiction, we assume the existence of $\overline{w}^{(n_0)}\neq 0$ for some $n_0 \in \mathbb{N}_\ast \setminus \mathbb{N} _{\sharp}.$  By the continuity of $T_{\eta^{-1}},$
		\begin{equation*}
			2 ^{-\frac{N-2m}{m}\overline{j}_k ^{(n)}} u_k (2 ^{-\overline{j}_k ^{(n)}}\cdot +\eta \overline{y}^{(n)}_k) = T_{\eta^{-1}} (2 ^{-\frac{N-2m}{m}\overline{j}_k ^{(n)}}  u_k (2 ^{-\overline{j}_k ^{(n)}} \cdot + \overline{y}_k ^{(n)} ) )\rightharpoonup \overline{w} ^{(n)} \circ \eta ^{-1} \text{ in }\mathcal{D}^{2,m}(\mathbb{R}^N).
		\end{equation*}
		Let $O_M = \{\eta _i\in \mathcal{O}(N)\setminus\{1\} :i=1,\ldots,M\}$ be an arbitrary distinct collection in $\mathcal{O}(N).$ Since $2^{\overline{j}_k ^{(n)}}|\overline{y}_k ^{(n)}|\rightarrow \infty,$ we have that $2^{\overline{j}_k ^{(n)}}|\eta_i \overline{y}_k ^{(n)} -\eta_j \overline{y}_k ^{(n)}| \rightarrow \infty,$ for all $ i \neq j.$ Consequently, from Lemma \ref{l_brezislieb_it} we get the following estimate, 
		\begin{equation*}
			\limsup_{k \rightarrow \infty}\|u_k\|^{m_\ast} _{m_\ast}  \geq \sum _{i=1}^M \|\overline{w}^{(n)} \circ \eta^{-1} _i\|^{m_\ast} _{m_\ast}= \sum _{i=1}^M \|\overline{w}^{(n)}\|^{m_\ast} _{m_\ast}= M \|\overline{w}^{(n)}\|^{m_\ast} _{m_\ast},
		\end{equation*}which is  a contradiction with the fact that $(u_k)$ is bounded in $\mathcal{D}^{2,m}(\mathbb{R}^N).$
	\end{proof}
	\subsection{Proof of Theorems \ref{teo_profcrit} and \ref{teo_profsub}} 
	The proofs are similar in spirit to  \cite[Theorems 2.6 and 7.3]{tintabanach} and \cite[Theorem 2.5]{tintacocompsub}.
	\begin{definition}\cite[Definition 2.1]{tintabanach}
		Let $Z$ be a Banach space, and let $D$ be a bounded set of bijective isometries on $Z$ containing the identity $I$ such that $D^{-1}= \{h^{-1}:h \in D \}$ is also a bounded set. The sequence $(u_k)\subset Z$ converges to zero $D$-weakly, and we write $u_k \stackrel{D}{\rightharpoonup}0,$ if $g_k ^{-1}u_k \rightharpoonup 0$ for every $(g_k) \subset D.$ We use the notation $u_k \stackrel{D}{\rightharpoonup}u,$ to mean that $u_k  - u \stackrel{D}{\rightharpoonup} 0.$
	\end{definition}
	We consider the groups
	\begin{align*}
		D_{\mathbb{R} ^N} &:= \left\lbrace d_{y,j} : \mathcal{D}^{2,m} (\mathbb{R}^N) \rightarrow \mathcal{D}^{2,m} (\mathbb{R}^N): \right.\\
		&\qquad \qquad \left. d_{y,j} u (x) := 2 ^{\frac{N-2m}{m} j} u (2 ^j (x-y)),\ y \in \mathbb{R}^N, \ j \in \mathbb{R} \right\rbrace,\\
		T_{\mathbb{R}^N } &:= \left\lbrace g_y : W^{2,m} (\mathbb{R} ^N) \rightarrow W^{2,m} (\mathbb{R} ^N) : g_y u(x) = u(x-y), \ y \in \mathbb{R}^N \right\rbrace,
	\end{align*}
	$
	D_{\mathbb{Z} ^N}  := \{ d_{y,j} \in D_{\mathbb{R} ^N} : y \in \mathbb{Z}^N, j \in \mathbb{Z} \}$ and 
	$ 
	T_{\mathbb{Z} ^N} := \{ g_y \in T_{\mathbb{R} ^N} : y \in \mathbb{Z}^N\}.
	$
	
	\begin{remark} Here, we have another natural question on our profile decomposition results.  
		Is it possible to replace the action of dilations $u \mapsto 2 ^{\frac{N-2m}{m} j} u (2^{j _k ^{(n)}} \cdot )$ by dilations like $u \mapsto \gamma ^{\frac{N-2m}{m} j} u (\gamma ^{j _k ^{(n)}} \cdot ),$ $\gamma >1,$ in Theorem \ref{teo_profcrit}?  Having an affirmative answer to this question,  one can consider more general self-similar nonlinearities, as explained in Sect. \ref{s_selfsimilar}. That question is equivalent to prove cocompactness type results for the group
		\begin{equation*}
			D^{(\gamma)}_{\mathbb{Z} ^N} := \left\lbrace d_{y,j} : \mathcal{D}^{2,m} (\mathbb{R}^N) \rightarrow \mathcal{D}^{2,m} (\mathbb{R}^N) : d_{y,j} u (x) := \gamma ^{\frac{N-2m}{m} j} u (\gamma ^j (x-y)),\ y \in \mathbb{R}^N, \ j \in \mathbb{R} \right\rbrace,
		\end{equation*}
		that is, if $u_k \stackrel{D^{(\gamma)}_{\mathbb{Z} ^N}}{\rightharpoonup} 0$ then $u_k \rightarrow 0$ in $L^{m_\ast}(\mathbb{R}^N).$
		See more details in \cite{tintabanach,tintageral}.
	\end{remark}
	
	Theorems \ref{teo_profcrit} and \ref{teo_profsub} follow by the next results about the behavior of actions of the groups $D_{\mathbb{R} ^N}$ and $T_{\mathbb{R}^N }.$
	\begin{lemma}\label{l_convaux}
		$\mathrm{(i)}$ Let $(y_k, j_k) \subset \mathbb{R}^N \times \mathbb{R},$ such that $(y_k, j_k) \rightarrow (y, j).$ Then $d_{y_k, j_k} u \rightarrow d_{y, j} u,$ in $\mathcal{D}^{2,m}(\mathbb{R}^N),$ for all $u \in \mathcal{D}^{2,m} (\mathbb{R}^N).$
		
		\noindent$\mathrm{(ii)}$ Let $u \in \mathcal{D}^{2,m} (\mathbb{R}^N) \setminus \{ 0 \}.$ The sequence $(d_{y_k,j_k}u )\subset \mathcal{D}^{2,m} (\mathbb{R}^N),$ with $(y_k,j_k) \subset \mathbb{R}^N \times \mathbb{R},$ converges weakly to zero if and only if $|j_k| + |y_k| \rightarrow \infty.$
		
		\noindent$\mathrm{(iii)}$ Let $(y_k)$ be a sequence in $\mathbb{R} ^N$ and $u \in W^{2,m}(\mathbb{R}^N)\setminus\{0\}.$ The sequence $( u (\cdot - y_k ) )$ converges weakly to zero in $W^{2,m}(\mathbb{R}^N)$ if, and only if $|y_k| \rightarrow \infty.$
	\end{lemma}
	\begin{proof}
		(i) By density and the invariance of the norm with respect to $D_{\mathbb{R} ^N},$ we may assume that $u \in C^\infty_0(\mathbb{R}^N).$ Since,
		\begin{equation*}
			\int_{\mathbb{R}^N} |\Delta( d_{y_k, j_k} u- d_{y, j} u)|^m\dx = \int_{\mathbb{R}^N} |\Delta u - 2 ^{\frac{N}{m}(j-j_k)} \Delta u(2 ^{j-j_k} \cdot + 2 ^j(y_k - y) )|^m \dx,
		\end{equation*}
		and the fact that $((y_k, j_k))$ is bounded,
		we may apply Lebesgue's dominated convergence theorem to conclude that $||\Delta (d_{y_k, j_k} u - d_{y, j} u)||_m \rightarrow 0 .$
		
		(ii) We argue by contradiction and assume that $|j_k| + |y_k|  \not \rightarrow \infty.$ So, up to subsequence, $j_k \rightarrow j$ and $y_k \rightarrow y.$ Hence by the above assertion, $d_{y_k, j_k} u \rightarrow d_{y, j} u$ in $\mathcal{D}^{2,m}(\mathbb{R}^N).$ Now consider the continuous linear functional
		\begin{equation*}
			\Phi_\varphi(u) = \int_{\mathbb{R}^N} |\Delta \varphi|^{m-2}\Delta \varphi \Delta u \dx,\ \varphi \in \mathcal{D}^{2,m}(\mathbb{R}^N).
		\end{equation*}
		Thus, $\Phi_\varphi(d_{y_k, j_k} ) \rightarrow \|\Delta d_{y, j} u \|^m_m = \|\Delta u \|_m ^m.$ This leads a contradiction with the hypothesis which implies that $\Phi_\varphi(d_{y_k, j_k} ) \rightarrow 0.$ Now suppose that $|j_k| + |y_k| \rightarrow \infty.$ Assuming $u \in C^\infty _0 (\mathbb{R}^N)$ we analyze two cases:

		\noindent  {\bf   Case 1:}
		$ j_k \rightarrow \infty $ or $j_k \rightarrow - \infty.$ Take $\Phi \in (\mathcal{D}^{2,m}(\mathbb{R}^N))^{\ast}.$ 
		Then,
		\begin{equation*}
			|\Phi (d_{y_k, j_k} )|^m \leq \| \Phi \| _{\ast} \| \Delta (d_{y, j} u )\|_m ^m \leq  2^{N j_k}\| \Phi \| _{\ast} \| \Delta u \| _\infty |\supp u|.
		\end{equation*}
		Consequently, if $j_k \rightarrow - \infty,$ then $\Phi (d_{y_k, j_k} ) \rightarrow 0.$ At the same time, if $j_k \rightarrow  \infty,$ using $\| \Delta (d_{y_k, j_k} u )\|^m_m = \| \Delta (d_{-2 ^{j_k}y_k, -j_k} u )\|^m_m$ and the same estimate above to conclude that $\Phi (d_{y_k, j_k} ) \rightarrow 0.$ \\
		\noindent {\bf Case 2:}  $j_k \rightarrow j$ and $|y_k| \rightarrow \infty.$ In this case we apply Lebesgue's dominated convergence theorem, since $d_{y_k, j_k} u(x) \rightarrow 0$  almost everywhere in $\mathbb{R}^N,$ and $|d_{y_k, j_k} u(x)|\leq C \| u\| _\infty \in L^1(\supp u).$

		(iii)  Minor changes in the preceding arguments for the functional 
		\begin{equation*}
			\Phi_\varphi(u) = \int_{\mathbb{R}^N} |\Delta \varphi|^{m-2}\Delta \varphi \Delta u +|\varphi|^{m-2}\varphi  u\dx,\ \varphi \in W^{2,m}(\mathbb{R}^N).
		\end{equation*}
		readily show this item. \end{proof}
	
	\begin{lemma}\label{l_brezislieb_it}
		Let $(u_k) \subset \mathcal{D}^{2,m}(\mathbb{R}^N)$ be a bounded sequence and $p=mN/(N-2m).$ Suppose that $(w ^{(n)}) _{n \in \mathbb{N}_{\ast } } \subset \mathcal{D}^{2,m}(\mathbb{R}^N),$ $(y_k ^{(n)}) \subset \mathbb{Z}^N,$ $(j _k ^{(n)})  \subset \mathbb{Z},$ $n \in \mathbb{N}_{\ast },$ satisfy \eqref{seis.um} and \eqref{seis.dois} in Theorem \ref{teo_profcrit}. Then, for every $M \in \mathbb{N}_\ast,$
		\begin{equation}\label{eq_toproveIT}
			\int _{\mathbb{R}^N} |u_k|^p\dx - \sum _{n=1}^M \int_{\mathbb{R}^N} |w^{(n)}|^p\dx - \int _{\mathbb{R}^N}\left|u_k - \sum _{n =1}^M 2 ^{\frac{N-2m}{m} j_k ^{(n)} } w^{(n)}(2^{j _k ^{(n)}} ( \cdot - y_k^{(n)} ) )\right|^p \dx\rightarrow 0.
		\end{equation}
		In particular, from \eqref{seis.um}, \eqref{seis.dois} and \eqref{seis.quatro} we obtain \eqref{seis.tres}.  
	\end{lemma}
	\begin{proof}
		From the compactness embedding $\mathcal{D}^{2,m}(\mathbb{R}^N) \hookrightarrow L ^\theta _{\loca} (\mathbb{R}^N),$ $1\leq \theta < m_\ast,$ and a
		diagonal argument applied on the subsequences of $(u_k)$ we can assume that \eqref{seis.um} converges almost everywhere in $\mathbb{R}^N,$ for all $n \in \mathbb{N}_\ast.$ 
		Thus, the case that $M=1,$ follows by a change of variables and  Br\'{e}zis-Lieb Lemma\cite{brezis-lieb}. Denoting
		\begin{equation*}
			v_k ^{(l)} = u_k - \sum _{n =1}^l 2 ^{\frac{N-2m}{m} j_k ^{(n)} } w^{(n)}(2^{j _k ^{(n)}} ( \cdot - y_k^{(n)} ) )\quad\text{and}\quad d_k^{(l)} = d _{y_k ^{(l)} , j _k ^{(l)}} u_k,
		\end{equation*}
		we see by Lemma \ref{l_convaux} and \eqref{seis.um} that $(d_k ^{(l+1)})^{-1}v_k ^{(l)}\rightharpoonup w^{(l+1)}$ in $\mathcal{D}^{2,m}(\mathbb{R}^N).$
		We have, once again by Br\'{e}zis-Lieb Lemma,
		\begin{equation}\label{eq_UM}
			\| (d_k^{(l+1)})^{-1} v_k ^{(l)} \|^p_p - \|(d_k^{(l+1)})^{-1} v_k ^{(l)} - w^{(l+1)} \|^p_p = \| w^{(l+1)} \|^p_p + o_k (1).
		\end{equation}
		Now suppose that
		\begin{equation}\label{eq_DOIS}
			\|u_k\|_p ^p - \sum _{n=1} ^l \| w^{(n)} \| _p ^p - \| v_k ^{(l)} \|_p ^p = o_k (1),
		\end{equation}
		holds, where $o_k (1)\rightarrow 0,$ as $k \rightarrow \infty.$ From \eqref{eq_UM} to get the value of $\|  v_k ^{(l)} \|^p_p = \| (d_k^{(l+1)})^{-1} v_k ^{(l)} \|^p_p$ and use this in \eqref{eq_DOIS}, to conclude that \eqref{eq_toproveIT} follows:
		\begin{equation*}
			\|u_k\|_p ^p - \sum _{n=1} ^{l+1} \| w^{(n)} \| _p ^p - \|(d_k ^{(l+1)})^{-1} ( v_k ^{(l)} - d_k ^{(l+1)} w^{(l+1)} )\|_p ^p = o_k (1).\qedhere
		\end{equation*}
	\end{proof}
	Next, we have some cocompactness results.
	\begin{proposition}
		$\mathrm{(i)}$ Let $(u_k)$ be a bounded sequence in $\mathcal{D}^{2,m}(\mathbb{R} ^N).$ Then $u_k \stackrel{D_{\mathbb{Z} ^N}}{\rightharpoonup} 0$ if and only if $u_k \rightarrow 0$ in $L^{m_\ast}(\mathbb{R}^N).$
		
		$\mathrm{(ii)}$ Let $(u_k)$ be a bounded sequence in $W^{2,m} (\mathbb{R}^N).$ Then $u_k \stackrel{T_{\mathbb{Z}^N }}{\rightharpoonup} 0$ in $W^{2,m} (\mathbb{R}^N),$ if and only if $u_k \rightarrow 0$ in $L^{\theta}(\mathbb{R}^N),$ for all $m<\theta<m _\ast.$
	\end{proposition}
	\begin{proof} To see (i) one just need to use \cite[Theorem 7.3]{tintabanach} or \cite[Sect. 3]{wavelet} and notice that if there exists $(y_k,j_k)$ such that $d_{y_k , j_k}u_k \rightharpoonup u \neq 0,$ then $\| u\| _{m_\ast} \leq \liminf_{k} \| d_{y_k , j_k}u_k\| _{m_\ast} = \lim_{k}\| u_k\|_{m_\ast}=0.$	Minor changes in the preceding arguments and by \cite[Theorem 2.5]{tintacocompsub} readily show that (ii) holds.
	\end{proof}
	\begin{proof}[Proof of Theorems \ref{teo_profcrit} and \ref{teo_profsub} completed] Follows from the previous results, \cite[Theorem 2.6]{tintabanach} and \cite[Theorem 3.3]{tintageral}. The last assertion follows from standard argument contained in the proof of \cite[Theorem 5.1]{tintabook}.\end{proof}
	\subsection{Compactness results about the quasilinear term}Let's recall the following result due to J. Simon, which is useful in the proof of Theorem \ref{teo_mincrit}.
	\begin{lemmaletter}\cite[Lemma 2.1]{simon}\label{l_simon}
		Let $x,y \in \mathbb{R}^N$ and $\left\langle \cdot, \cdot  \right\rangle $ the standard scalar product in $\mathbb{R}^N.$ Then there exists $C>0$ such that
		\begin{equation*}
			\left\langle |x|^{p-2}x - |y|^{p-2}y,x-y \right\rangle \geq \left\lbrace 
			\begin{aligned}
				&C|x-y|^p, &&\text{ if }p \geq 2,\\
				&C|x-y|^2 (|x| + |y|)^{2-p},&&\text{ if }1<p<2.
			\end{aligned}
			\right.
		\end{equation*}
	\end{lemmaletter}
	\begin{proposition}\label{p_aeconv}
		Define for $u \in \mathcal{D}^{2,m}(\mathbb{R}^N),$
		\begin{equation*}
			\mathcal{N}(u)=\int_{\mathbb{R}^N} b(x) H\left((b(x) )^{-1}L (u)\right)\dx  \quad \text{and}\quad \Phi(u)=\int_{\mathbb{R}^N}F(x,u)\dx.
		\end{equation*}
		Let $(\mu _k)$ be a bounded sequence of real numbers. Suppose that either one of the following conditions holds:
		
		$\mathrm{(i)}$ \eqref{HYPcrit}, $b(x)\equiv 1,$ \eqref{g_dois}, \eqref{g_tres}, \eqref{constant_ss}, $V(x) \equiv -\lambda |x|^{-2}$ with $0 \leq \lambda < \Lambda _{N,m}^{1/m};$ and $u_k\rightharpoonup u\neq 0$ in $\mathcal{D}^{2,m} _{\rad }(\mathbb{R}^N);$
		
		$\mathrm{(ii)}$ \eqref{HYPsub}, \eqref{g_dois}, \eqref{g_tres}, \eqref{bem_def}, \eqref{V_sirakov}, \eqref{V_pesocerto}; and $u_k\rightharpoonup u\neq 0$ in $W_V^{2,m}(\mathbb{R}^N);$
		Then $\mathcal{N}'(u_k)- \mu _k \Phi'(u_k) \rightarrow 0,$ implies $L(u_k)(x) \rightarrow L(u)(x)$ almost everywhere in $\mathbb{R}^N.$
	\end{proposition}
	\begin{proof}
		(i) We consider first the case where $(u_k)\subset \mathcal{D}^{2,m} _{\rad }(\mathbb{R}^N),$ the nonlinearities with critical growth in the hyperbola sense and $V(x) \equiv -\lambda |x|^{-2}.$ We prove that
		\begin{equation}\label{aecapeta}
			L(u_k)(x) \rightarrow L(u)(x) \text{  almost everywhere in }\mathbb{R}^N.
		\end{equation}
		Given $R>\delta>0,$ let $\mathcal{S}_{0} \in C^\infty(\mathbb{R})$ and $\xi _{R, \delta } \in C^\infty _{0, \rad} (B_{2R})$ such that 
		\begin{equation*}
			\mathcal{S}_{0}(t)  = 
			\left\{
			\begin{aligned}
				& t,\text{ for }|t| \leq 1,\\
				& 1, \text{ for }|t| \geq 2,
			\end{aligned}
			\right.
			\quad\text{and}\quad
			\xi _{R, \delta } (x) = 
			\left\{
			\begin{aligned}
				& 1,\text{ if }\delta \leq |x| \leq R,\\
				& 0, \text{ if }|x| \geq 2R\text{ or }|x|\leq \delta / 2.
			\end{aligned}
			\right.
		\end{equation*}
		For $\sigma >0,$ define $D_\sigma (t) = \mathcal{S}_{0}(t/\sigma).$
		Notice that for $v \in \mathcal{D}^{2,m} _{\rad }(\mathbb{R}^N) $, the identity 
		\[
		\Delta (D_\sigma (\xi _{R, \delta } v)) = D'_\sigma (\xi _{R, \delta } v)\Delta (\xi _{R, \delta } v)+  D'' _\sigma (\xi _{R, \delta } v)|\nabla (\xi _{R, \delta }v)|^2,
		\]
		and Lemma \ref{l_radial} implies that $D_\sigma (\xi _{R, \delta } v)\in \mathcal{D}^{2,m} _{\rad }(\mathbb{R}^N),$ with
		\begin{align}\label{estimativa_capeta}
			\int _{\mathbb{R}^N}|\Delta (D_\sigma (\xi _{R, \delta } v))|^m \dx  & \leq \frac{2^{m-1}}{\sigma ^m}\| S'_0 \|^m_\infty  \int _{A_{R,\delta }}  |\Delta (\xi _{R, \delta } v) | ^m\dx \nonumber \\
			&\quad+\frac{2^{m-1}}{\sigma ^{2m}}\| S''_0 \|^m_\infty\int _{A_{R,\delta } }|\nabla (\xi _{R, \delta }v)|^{2m} \dx  ,
		\end{align}
		where $A_{R,\delta } = B_{2R} \setminus B_{\delta / 2}.$ Now we compute
		\begin{multline}\label{exp_grande}
			(\mathcal{N}'(u_k)- \mu _k \Phi'(u_k) ) \cdot (D_\sigma(\xi _{R, \delta }(u_k -u )) ) \\=\int_{\mathbb{R}^N}\left[h(L(u_k))-h(L(u))\right] L (D_\sigma(\xi _{R, \delta }(u_k -u ))  )\dx\\
			+ \int_{\mathbb{R}^N}h(L(u) )L(D_\sigma(\xi _{R, \delta }(u_k -u )) ) \dx - \mu _k \int_{\mathbb{R}^N} f(u_k) D_\sigma(\xi _{R, \delta }(u_k -u )) \dx.
		\end{multline}
		By the compact Sobolev embedding $\mathcal{D}^{2,m}(\mathbb{R}^N)\hookrightarrow L ^{\theta}_{\loca} (\mathbb{R}^N),$ for $1\leq \theta <mN/(N-2m),$ Lemma \ref{l_radial} and Lebesgue's theorem in $L^1(A_{R,\delta }),$ up to subsequence, 
		\begin{equation*}
			\left|  \int_{\mathbb{R}^N} f(u_k) D_\sigma(\xi _{R, \delta }(u_k -u )) \dx \right| \leq \mathcal{C}_{\ast}\| u_k \| _p ^{p-1}\| D_\sigma(\xi _{R, \delta }(u_k -u ))  \|_{L^p(A_{R,\delta })}\rightarrow 0,\ k \rightarrow \infty,
		\end{equation*}
		The compact embedding $\mathcal{D}^{1,m}(\mathbb{R}^N)\hookrightarrow L ^{\theta}_{\loca} (\mathbb{R}^N),$ for $1\leq \theta <mN/(N-m),$ and the fact that $\partial u_k / \partial x_j \rightharpoonup \partial u / \partial x_j$ in $ \mathcal{D}^{1,m}(\mathbb{R}^N)$, up to subsequence, implies
		\begin{equation*}
			\int_{U^\sigma _k }\left[h(L(u_k))-h(L(u))\right] L (\xi _{R, \delta }(u_k -u )  )\dx  = \int_{U^\sigma _k }\left[h(L(u_k))-h(L(u))\right] L (u_k -u)\xi _{R, \delta }\dx+o_k(1),
		\end{equation*}
		where $U^\sigma _k =  \{ x \in A_{R,\delta }  :  |u_k(x) - u (x)| \leq  \sigma \}$ and $o_k (1)\rightarrow 0,$ as $k \rightarrow \infty,$ independent of $\sigma.$ Consequently,
		\begin{multline}\label{limitfinal}
			\int_{U^\sigma _k }\left[h(L(u_k))-h(L(u))\right] L (u_k -u)\xi _{R, \delta }\dx \\ = o_k(1)  - \int_{A_{R,\delta } }h(L(u) )L(D_\sigma(\xi _{R, \delta }(u_k -u )) ) \dx \\- \int_{\mathbb{R}^N \setminus U^\sigma _k}\left[h(L(u_k))-h(L(u))\right] L (D_\sigma(\xi _{R, \delta }(u_k -u ))  )\dx.
		\end{multline}
		Moreover, by \eqref{g_tres} and Lemma \ref{l_simon}, 
		\begin{equation*}
			0 \leq \int_{U^1 _k }\left[h(L(u_k))-h(L(u))\right] L (u_k -u)\xi _{R, \delta }\dx \leq \int_{U^\sigma _k }\left[h(L(u_k))-h(L(u))\right] L (u_k -u)\xi _{R, \delta }\dx.
		\end{equation*}
		On the other hand, by \eqref{estimativa_capeta} and Proposition \ref{p_norma_comparar} we find that,
		\begin{align*}
			&\limsup_k \left|     \int_{A_{R,\delta } }h(L(u) )L(D_\sigma(\xi _{R, \delta }(u_k -u )) ) \dx \right|,  \\
			&\limsup_k  \left| \int_{\mathbb{R}^N \setminus U^\sigma _k}\left[h(L(u_k))-h(L(u))\right] L (D_\sigma(\xi _{R, \delta }(u_k -u ))  )\dx \right| \leq \frac{C}{\sigma ^m},\ \forall \, \sigma \geq 1.
		\end{align*}
		Therefore, we can take the limit $k\rightarrow\infty$ in \eqref{limitfinal}, to get that
		\begin{equation*}
			\lim_{k\rightarrow \infty}\int_{ A_{R,\delta } }\left[h(L(u_k))-h(L(u))\right] L (u_k -u)\xi _{R, \delta }\mathcal{X}_{U^1 _k}\dx=0.
		\end{equation*}
		Since $\mathcal{X}_{U^1 _k}(x) \rightarrow 1$ almost everywhere in $A_{R,\delta },$ by Lemma \ref{l_simon}, we conclude that $L(u_k)(x) \rightarrow L(u)(x)$ in $A_{R,\delta }$. Since $\delta, R$ are arbitrarily, \eqref{aecapeta} follows.
		
		(ii) The case where the nonlinearities have subcritical growth in the hyperbola sense follows by an analogous argument. In fact, we consider $\mathcal{S}_{0}\equiv t,$ $\sigma =1$ and $\delta =0$ in the above cutoff functions. From Corollary \ref{l_norma_comparar_sub}, the functional $v \mapsto \int_{\mathbb{R}^N}h(L(u) )L(\xi _Rv ) \dx$ is linear and continuous in $W_V^{2,m}(\mathbb{R}^N),$ because
		\begin{equation*}
			\left| \int_{\mathbb{R}^N}h(L(u) )L(\xi _Rv ) \dx \right| \leq C \| L(u) \| _m ^{m-1}\left[C_1\| \Delta \xi _{R, 0 }\|_{N/2}^m  + C_2\| \nabla \xi _{R, 0 } \|^m_N +1\right]\|\Delta v\| _m ^m,	\end{equation*}
		for all $v \in W_V^{2,m}(\mathbb{R}^N).$ Applying Lebesgue's theorem in the last term of \eqref{exp_grande} and using $\partial u_k / \partial x_j \rightharpoonup \partial u / \partial x_j$ in $ W_V^{1,m}(\mathbb{R}^N)$ $j=1,\ldots,N,$ up to subsequence, 
		\begin{equation*}
			\lim_{k \rightarrow \infty}\int_{\mathbb{R}^N}\left[h(L(u_k))-h(L(u))\right] L ((u_k -u )  ) \xi _{R, 0 }\dx =0.
		\end{equation*}
		Thus, the desired conclusion follows, once again, by \eqref{g_tres} and Lemma \ref{l_simon}.\end{proof}
	\subsection{Behavior of weak profile decomposition convergence under nonlinearities}
	We describe the limit of 
	$
	u \mapsto \int_{\mathbb{R}^N} F(x,u)\dx,
	$
	acting on bounded sequences $(u_k)$ by using the profile decomposition of  Theorems \ref{teo_profcrit} and \ref{teo_profsub}. 
	The following results are extensions of Br\'{e}zis-Lieb Lemma.
	\begin{proposition}\label{prop_weakcrit}
		$\mathrm{(i)}$ Let $(u_k)\subset\mathcal{D}^{2,m} (\mathbb{R}^N)$ be bounded and $(w ^{(n)}) _{n \in \mathbb{N}_{\ast } }$ in $\mathcal{D}^{2,m} (\mathbb{R}^N),$ $n \in \mathbb{N}_{\ast },$ given by Theorem \ref{teo_profcrit}. If $F(t)$ satisfies \eqref{original_selfsimilar} and is locally Lipschitz then, up to a subsequence,
		\begin{equation*}
			\lim _{k \rightarrow \infty} \int _{\mathbb{R}^N } F(u_k) \dx = \sum _{n \in \mathbb{N} _{\ast }} \int _{\mathbb{R}^N} F(w^{(n)}) \dx.
		\end{equation*}
		$\mathrm{(ii)}$ Let $(u_k)\subset W^{2,m} (\mathbb{R}^N) $ be a bounded and $(w ^{(n)}) _{n \in \mathbb{N}_{0} }$ in $W^{2.m}(\mathbb{R}^N)$ given by Theorem \ref{teo_profsub}. If $f(x,t)$ is $\mathbb{Z}^N$--periodic and satisfies \eqref{bem_def}, then, up to a subsequence,
		\begin{equation*}
			\lim _{k \rightarrow \infty}\int _{\mathbb{R}^N} F(x, u_k) \dx = \sum _{n \in \mathbb{N}_0} \int _{\mathbb{R} ^N} F (x,w^{(n)}) \dx.
		\end{equation*}
		$\mathrm{(iii)}$ If \eqref{bem_def} and \eqref{V_sirakov} hold, then for any $(u_k) \subset W^{2.m}_V(\mathbb{R}^N)$ bounded such that $u_k \rightarrow u$ in $L^\theta (\mathbb{R}^N),$ for some $\theta \in (m, m _\ast),$ up to subsequence, we have
		\begin{equation*}
			\lim _{k \rightarrow \infty}\int _{\mathbb{R}^N} f(x,u_k) u_k \dx = \int _{\mathbb{R}^N} f(x,u)u \dx.
		\end{equation*}
		Moreover, if $(v_k)$ is bounded in $W^{2.m} _V (\mathbb{R}^N)$ with $u_k - v_k \rightarrow 0$ in $L^\theta (\mathbb{R}^N),$ for some $m< \theta <m_\ast,$ then, up to a subsequence,
		\begin{equation*}
			\lim _ {k\rightarrow \infty} \int _{\mathbb{R}^N} F(x,u_k) - F(x,v_k) \dx=0.
		\end{equation*}
	\end{proposition}
	\begin{proof} Follows by a similar method of \cite[Theorem 4.1 and Lemma 5.5]{tintabook}.
	\end{proof}
	\begin{corollary}\label{brezis-lieb}
		Assume either that $u_k \rightharpoonup u$ in $\mathcal{D}^{2,m} (\mathbb{R}^N)$ and let $F(t)$  be as in Proposition \ref{prop_weakcrit}--(i) then, up to a subsequence,
		\begin{equation*}
			\lim _{k\rightarrow \infty}\int _{\mathbb{R}^N}  F(u_k) - F(u - u_k) - F(u)  \dx = 0.
		\end{equation*}
		The same conclusion holds if $u_k \rightharpoonup u$ in $W^{2,m} (\mathbb{R}^N)$ and $F(x,t)$ satisfies the assumptions of  Proposition \ref{prop_weakcrit}--(ii).
	\end{corollary}
	
	\section{Proof of Theorem \ref{teo_mincrit}}\label{s_t_mincrit}
	\begin{proof}[Proof of Theorem \ref{teo_mincrit}] Using Sobolev embedding $\mathcal{D}^{2,m}(\mathbb{R}^N) \hookrightarrow L ^p (\mathbb{R}^N)$ together with \eqref{constant_ss}, we obtain that $\Theta_{p,q,\lambda} >0.$  Consider the functionals 
		\begin{equation*}
			\mathcal{N}_\lambda(u)=\int_{\mathbb{R}^N}|L_\lambda (u)|^m\dx\quad \text{and}\quad\Phi(u)=\int_{\mathbb{R}^N}F(u)\dx,\ u \in \mathcal{D}^{2,m}_{\rad}(\mathbb{R}^N).
		\end{equation*}
		Applying Ekeland variational principle \cite[Theorem 3.1 and Corollary 3.4]{ekeland}, we have sequences $( u_k) \subset \mathcal{D}^{2,m}_{\rad}(\mathbb{R}^N)$ and $(\mu _k ) \subset \mathbb{R}$ such that  $\Phi(u_k)=1 $ and $\mathcal{N}'_\lambda(u_k) - \mu _k \Phi'(u_k) \rightarrow 0\text{ in } (\mathcal{D}^{2,m}_{\rad}(\mathbb{R}^N))',$  and $\mathcal{N}_\lambda(u_k) \rightarrow \Theta_{p,q,\lambda}.$
		By Proposition \ref{p_norma_comparar}, we have that $(u_k)$ is bounded in $\mathcal{D}^{2,m}_{\rad}(\mathbb{R}^N).$ Let $(w^{(n)})_{n \in \mathbb{N}_{\ast}}$ and $(j^{(n)}_k),$ $n \in \mathbb{N}_{\ast},$ the collection of profiles given in Theorem \ref{teo_profcrit} and Proposition \ref{prop_radial}. Applying Proposition \ref{prop_weakcrit},
		$
		\sum_{n \in \mathbb{N}_\ast} \int_{\mathbb{R}^N}F(w^{(n)}) \dx =1.
		$
		Thus, there is at least one $n_0 \in \mathbb{N}_\ast$ such that $0 < \int_{\mathbb{R}^N}F(w^{(n_0)}) \dx \leq 1.$ If equality occurs, we have
		\begin{equation*}
			\Theta_{p,q,\lambda} \leq \int_{\mathbb{R} ^N} |L_\lambda(w^{(n_0)}) |^m \dx \leq \liminf _{k \rightarrow \infty}  \int_{\mathbb{R} ^N} |L_\lambda (d_k^{-1} u_k) |^m\dx = \Theta_{p,q,\lambda},
		\end{equation*} 
		where $d_k^{-1} u_k = 2 ^{-\frac{N-2m}{m}j_k ^{(n_0)}}u _k (2 ^{- j_k ^{(n_0)}} \cdot ).$ In the previous argument we have used the invariance of $L_\lambda(u)$ with respect to dilations given in Theorem \ref{teo_profcrit}. Since $\mathcal{D}_{\rad}^{2,m}(\mathbb{R}^N)$ equipped with the norm $\| \cdot \|_{\lambda}$ is uniformly convex, we have that $w_k := d_k^{-1} u_k \rightarrow w^{(n_0)}$ in $\mathcal{D}^{2,m}_{\rad}(\mathbb{R}^N)$ and $w:=w^{(n_0)}$ is a minimizer for \eqref{min}. Let us suppose that $0 < \delta:=\int_{\mathbb{R}^N}F(w^{(n_0)}) \dx < 1.$ Using \eqref{H_um}, we see that
		\begin{equation*}
			\mathcal{N}'_\lambda(w_k)\cdot v - \mu _k \Phi'(w_k) \cdot v=\mathcal{N}'_\lambda(u_k)\cdot (d_k^{(n_0)}v) - \mu _k \Phi'(u_k) \cdot (d_k^{(n_0)}v),
		\end{equation*}
		$v \in \mathcal{D}_{\rad}^{2,m}(\mathbb{R}^N),$ where $d_k^{(n_0)}v = 2 ^{\frac{N-2m}{m}j_k ^{(n_0)}}v (2 ^{j_k ^{(n_0)}} \cdot ),$ and we conclude that $\mathcal{N}'_\lambda(w_k) - \mu _k \Phi'(w_k) \rightarrow 0$ in $(\mathcal{D}^{2,m}_{\rad}(\mathbb{R}^N))'.$ Moreover, $\mathcal{N}_\lambda(w_k) \rightarrow \Theta_{p,q,\lambda}$ and $\Phi(w_k)=1.$ Now, writing $v_k = w_k - w,$ by Corollary \ref{brezis-lieb},
		\begin{equation}\label{B-L_convergence}
			1-\int _{\mathbb{R}^N}  F(w) \dx = \lim _{k\rightarrow \infty}\int _{\mathbb{R}^N} F(v_k)\dx.
		\end{equation}
		Let $\hat{v}_k = v_k (|1- \delta|^{1/N } \beta _k ^{1/N }\cdot ),$ where $\beta _k := \int _{\mathbb{R}^N} F(v_k (|1- \delta |^{1/N} \cdot )) \rightarrow 1,$ because of \eqref{B-L_convergence}. We also have
		\begin{equation}\label{segundo_termo}
			\int _{\mathbb{R}^N} F(\hat{v}_k) \dx =1,
		\end{equation}
		which implies 
		\[
		\Theta_{p,q} \left||1-\delta |\beta _k\right|^{1-\frac{2m}{N}}\leq \int_{\mathbb{R}^N}|L (v_k)|^m \dx.
		\]
		Moreover for $\hat{w} := w(\delta ^{1/N} \cdot ),$ there holds $\int _{\mathbb{R}^N} F(\hat{w}) \dx =1,$	and so
		\begin{equation}\label{primeiro_termo}
			\Theta_{p,q} \left[ \int_{\mathbb{R}^N}F(w) \dx \right]^{1-\frac{2m}{N}}  \leq \int_{\mathbb{R}^N} |L (w)|^m \dx.
		\end{equation}
		By Br\'{e}zis-Lieb Lemma \cite{brezis-lieb} and Proposition \ref{p_aeconv}, we see that
		\begin{equation}\label{limite_infimoBL}
			\begin{aligned}
				\Theta_{p,q,\lambda} &= \int_{\mathbb{R}^N}|L_\lambda(w_k)|^m \dx + o_k(1)\\
				&=  \int_{\mathbb{R}^N} |L_\lambda(v_k)|^m\dx + \int_{\mathbb{R}^N}|L_\lambda(w)|^m \dx +o_k(1).
			\end{aligned}
		\end{equation}
		Replacing \eqref{segundo_termo} and \eqref{primeiro_termo} in \eqref{limite_infimoBL} and passing the limit, we obtain $1 \geq \left| \delta \right|^{1-\frac{2m}{N}} + |1-\delta |^{1-\frac{2m}{N}},$ which is a contradiction with the fact that $0<\delta <1.$
		
		Thus $w\in \mathcal{D}^{2,m}_\rad (\mathbb{R}^N)$ is a minimizer in \eqref{min} and consequently we have
		\begin{equation*}
			m\int _{\mathbb{R}^N} |L_\lambda (w)|^{m-2} L_\lambda (w) L_\lambda(\varphi) \dx = \mu \int _{\mathbb{R}^N} f(w) \varphi \dx, \; \forall \varphi \in \mathcal{D}^{2,m}_\rad (\mathbb{R}^N),
		\end{equation*}
		where $\mu \in \mathbb{R}$ is a Lagrange multiplier. Taking $\varphi=w$ in the above identity we have $\mu \neq 0.$ Using the principle of symmetric criticality \cite{Palais}, we have that  $w$ is a weak solution of Eq. \eqref{P} for the nonlinearity $(\mu/m) f(t)$, and applying Corollary \ref{cor_pohozaev} we get  $\mu = (m/p) \Theta_{p,q,\lambda}.$ Consequently,  $u = w( \cdot  /\beta)$ is a nontrivial weak solution of Eq.~\eqref{P}, where $\beta = (\mu/m) ^{1/(2m)}= (\Theta_{p,q,\lambda} / p ) ^{1/(2m)}.$ 
		
		Let us prove now that $u=w(\cdot / \beta )$ is a ground state solution of Eq.~\eqref{P}. We start by applying  Corollary \ref{cor_pohozaev}  to obtain
		\begin{equation}\label{gs1}
			I(u) = \left(\frac{1}{m} -\frac{1}{p}\right)\int _{\mathbb{R}^N} |L_\lambda (u)|^{m} \dx =\frac{p-m}{mp} p ^{-\frac{N-2m}{2m}}\| L_\lambda (w) \| _m^{N/2}.
		\end{equation}
		Let $\psi \in \mathcal{D} ^{2,m} (\mathbb{R}^N)$ be a nontrivial weak solution of Eq.~\eqref{P}.  Setting $\psi _\sigma = \psi ( \cdot / \sigma)$, where $\sigma>0$ is such that 
		$
		\int_{\mathbb{R}^N} F ( \psi _\sigma) \dx = 1.
		$
		Applying Corollary \ref{cor_pohozaev} for $\psi,$ we have  $\sigma  = p^{1/N} \| L_\lambda(\psi) \|_m ^{-m/N},$ and thus 
		\[
		\| L_\lambda(\psi _\sigma ) \|_m^m = p^{(N-2m)/N} \| L_\lambda(\psi) \|_m ^{m(2m/N)}.
		\]
		Therefore,
		\begin{equation}\label{gs2}
			I(\psi)=\left(\frac{1}{m} -\frac{1}{p}\right)\int _{\mathbb{R}^N} |L_\lambda (\psi)|^{m} \dx =\frac{p-m}{mp} p ^{-\frac{N-2m}{2m}}\| L_\lambda (\psi _\sigma) \| _m^{N/2}.
		\end{equation}
		From \eqref{gs1} and \eqref{gs2}, we conclude that $I(u) \leq I(\psi),$ that is, $u$ is a ground state solution for Eq.~\eqref{P}. The remainder of the proof follows by Theorem \ref{p_regcrit}.
	\end{proof}
	\section{Proof of Theorems \ref{GS_subcrit} and \ref{GS_crit}}\label{s_t_proof}
	\begin{proof}[Proof of Theorem \ref{GS_crit}]
		By Proposition \ref{p_geoMP} there exists a bounded Palais-Smale sequence at the mountain pass level: $(u_k)\subset \mathcal{D}^{2,m}_{\rad}(\mathbb{R}^N)$ such that $I(u_k) \rightarrow c(I)$ and $I'(u_k) \rightarrow 0.$ Consider $(w^{(n)})_{n \in \mathbb{N}_{\ast}}$ and $(j^{(n)}_k),$ $n \in \mathbb{N}_{\ast},$ given in Theorem \ref{teo_profcrit} and Proposition \ref{prop_radial}. If $w^{(n)}= 0$ for all $n\in \mathbb{N}_{\ast},$ then by assertion \eqref{seis.quatro}, up to subsequence, $u_k \rightarrow 0$ in $L^{p} (\mathbb{R}^N).$ Consequently, by Propositions \ref{prop_weakcrit}, we have
		\begin{equation}\label{eq_GS_primeira}
			\begin{aligned}
				o_k(1)+c(I)&=I(u_k) = \int_{\mathbb{R}^N}H(L_\lambda(u_k) ) \dx - \int_{\mathbb{R}^N} F(u_k) \dx \\
				&= \int_{\mathbb{R}^N}H(L_\lambda(u_k) ) \dx + o_k(1),  \\
				o_k(1)&=I'(u_k)\cdot u_k =\int_{\mathbb{R}^N}h(L_\lambda(u_k) ) L_\lambda(u_k) \dx + o_k(1),
			\end{aligned}
		\end{equation}
		a contradiction, since $c(I)>0.$ Hence there at least one $w^{(n)} \neq 0,$ $n \in \mathbb{N}_\ast.$ Moreover, each nonzero $w^{(n)}$ is a critical point of $I.$ To see this, let $d_k^{(n)}\varphi = 2 ^{\frac{N-2m}{m}j_k ^{(n_0)}}\varphi (2 ^{j_k ^{(n_0)}} \cdot )$ and $\hat{d}^{(n)}_k\varphi = 2 ^{-\frac{N-2m}{m}j_k ^{(n)}}\varphi(2 ^{- j_k ^{(n)}} \cdot ),$ $\varphi\in \mathcal{D}_\rad^{2,m}(\mathbb{R}^N.$ Using \eqref{g_quatro} and \eqref{H_um} we have
		\begin{equation*}
			I' (u_k) \cdot (d_k^{(n)}\varphi) = \int _{\mathbb{R}^N}h\left(L_\lambda(\hat{d}^{(n)}_k u_k)\right)L_\lambda (\varphi)\dx - \int _{\mathbb{R}^N} f(\hat{d}^{(n)}_k u_k)\varphi\dx,\
		\end{equation*}
		for all $\varphi \in \mathcal{D}_\rad^{2,m}(\mathbb{R}^N).$ By Egoroff Theorem and Proposition \ref{p_aeconv} we see that $h\left(L_\lambda(\hat{d}^{(n)}_k u_k)\right) \rightharpoonup h\left(L_\lambda(w^{(n)})\right)$ in $L^{m/(m-1)}(\mathbb{R}^N),$ up to subsequence. Furthermore, since  $\Phi _\varphi : L^{m/(m-1)}(\mathbb{R}^N) \rightarrow \mathbb{R},$ $\varphi \in \mathcal{D}_\rad^{2,m}(\mathbb{R}^N),$ given by
		\begin{equation*}
			\Phi_\varphi (u) = \int_{\mathbb{R}^N}u L_\lambda (\varphi) \dx,
		\end{equation*}
		is linear and continuous, we conclude that
		\begin{equation*}
			\int _{\mathbb{R}^N}h\left(L_\lambda(\hat{d}^{(n)}_k u_k)\right)L_\lambda (\varphi)\dx \rightarrow \int _{\mathbb{R}^N}h\left(L_\lambda(w^{(n)})\right)L_\lambda (\varphi)\dx,\ \forall \varphi \in \mathcal{D}_\rad^{2,m}(\mathbb{R}^N).
		\end{equation*}
		Now, Lebesgue's dominated convergence theorem and compactness of the embedding $\mathcal{D}^{2,m}(\mathbb{R}^N) \hookrightarrow L^p _{\loca} (\mathbb{R}^N)$ yield, up to subsequence,
		\begin{equation*}
			\int _{\mathbb{R}^N} f(\hat{d}^{(n)}_k u_k)\varphi\dx \rightarrow \int _{\mathbb{R}^N} f(w^{(n)})\varphi\dx,\ \forall \, \varphi \in C^\infty _0 (\mathbb{R}^N).
		\end{equation*}
		Since $I' (u_k) \cdot (d_k^{(n)}\varphi) \rightarrow 0,$ $\varphi \in \mathcal{D}_\rad^{2,m}(\mathbb{R}^N),$ the fact that $w^{(n)}$ is a critical point of $I$ follows. In particular, we get that
		\begin{equation*}
			\mathcal{G}_{\mathcal{S}} = \inf \left\lbrace  I(u) : u\in \mathcal{D}_\rad^{2,m}(\mathbb{R}^N)\setminus \{0 \},\ I'(u) = 0 \right\rbrace\geq 0.
		\end{equation*}
		We are going to prove that is $\mathcal{G}_{\mathcal{S}}$ is positive and is attained. Let  $(u_k)$ be a  minimizing sequence of $\mathcal{G}_{\mathcal{S}},$ that is $I(u_k) \rightarrow \mathcal{G}_{\mathcal{S}}$ and $I'(u_k)=0.$ By Proposition \ref{p_geoMP},   $(u_k)$ is bounded. If $\mathcal{G}_{\mathcal{S}} = 0 ,$ then the argument of Proposition \ref{p_geoMP} also implies that $\|L_\lambda (u_ k)\|_{m} ^m = o_k(1).$ In this case, using \eqref{eq_GS_primeira} we conclude that 
		\begin{equation}\label{eq_obvio}
			\int_{\mathbb{R}^N}h(L_\lambda(u_k) )L_\lambda(u_k) \dx = o_k (1),
		\end{equation}
		and at the same time there exists $C_1>0$ such that 
		\begin{align}\label{eqreplace}
			\int _{\mathbb{R}^N}  h(L_\lambda (u_k)) L_\lambda (u_k)  \dx &= \int _{\mathbb{R}^N}f(u_k)u_k \dx\\
			&\leq C \| u_k\|_p ^p \leq C_1 \left[\int _{\mathbb{R}^N}|L(u_k)|^m \dx \right]^{\frac{p}{m}}.\nonumber 
		\end{align}
		Consequently, condition \eqref{g_dois} leads to the following contradiction with \eqref{eq_obvio},
		\begin{equation*}
			1 \leq C_1 \mathcal{C}_1^{-p/m} \left[ \int _{\mathbb{R}^N} h(L_\lambda (u_k)) L_\lambda (u_k) \dx \right] ^{\frac{p}{m}-1}.
		\end{equation*}
		Now let $(w^{(n)})_{n \in \mathbb{N}_{\ast}}$ and $(j^{(n)}_k),$ $n \in \mathbb{N}_{\ast},$ the profile of weak convergence in $\mathcal{D}^{2,m}_\rad (\mathbb{R}^N)$ for $(u_k).$ It suffices to prove that there exists $w^{(n_0)} \neq 0.$ In fact, the same argument above shows that $w^{(n_0)}$ is a critical point of $I.$ Moreover,   by Corollary \ref{cor_pohozaev} and Fatou Lemma, we have
		\begin{align}
			\mathcal{G}_{\mathcal{S}} = \lim _{k \rightarrow\infty} I(u_k) &= \liminf _{k \rightarrow\infty} \left[ \frac{2}{N}\int _{\mathbb{R}^N}  h(L_\lambda (u_k)) L_\lambda (u_k)  \dx \right] \nonumber\\
			&\geq \frac{2}{N}\int _{\mathbb{R}^N}  h(L_\lambda (w^{(n_0)})) L_\lambda (w^{(n_0)})  \dx = I(w^{(n_0)})\label{eq_replace}.
		\end{align}
		In above estimate we have used \eqref{g_dois} to ensure that $h(L_\lambda (u_k)) L_\lambda (u_k) \geq 0$ almost everywhere in $\mathbb{R}^N.$
		Let us  assume by contradiction that $w^{(n)}=0,$ for all $n\in \mathbb{N}_\ast.$ We can apply the same argument in \eqref{eq_GS_primeira} to obtain a contradiction with $\mathcal{G}_{\mathcal{S}}>0.$ The remainder of the proof follows from  Theorem \ref{p_regcrit}.
	\end{proof}
	\begin{proof}[Proof of Theorem \ref{GS_subcrit}]  
		By a similar argument used to prove Theorem~\ref{GS_crit}, with some modifications. Precisely, replacing the action of dilations in $\mathcal{D}_\rad ^{2,m} (\mathbb{R}^N)$ by translations in $W_V^{2,m}(\mathbb{R}^N).$
		
		(i) The existence of strong solution of \eqref{P} is obtained by nonzero profiles $w^{(n)},$ $n \in \mathbb{N}_\ast,$ given in Theorem \ref{teo_profcrit}. 
		This follows by the same method as in  Theorem~\ref{GS_crit}  replacing $d_k^{(n)}\varphi,$ $\hat{d}^{(n)}_k\varphi $ and $L_\lambda,$ by $g_k^{(n)}\varphi = \varphi (\cdot - y_k ^{(n)}),$ $\hat{g}_k^{(n)}\varphi = \varphi (\cdot + y_k ^{(n)})$ and $L = -\Delta + V(x),$ respectively, with $\varphi \in W^{2,m}(\mathbb{R}^N).$ 
		
		(ii) From the previous case (i), we can consider
		\begin{equation*}
			\mathcal{G}_{\mathcal{S}} = \inf \left\lbrace  I(u) : u\in W^{2,m}(\mathbb{R}^N)\setminus \{0 \},\ I'(u) = 0 \right\rbrace\geq 0.
		\end{equation*}
		We now use the following estimate
		\begin{equation*}
			\int _{\mathbb{R}^N}  h(L (u_k)) L (u_k)  \dx = \int _{\mathbb{R}^N}f(x,u_k)u_k \dx\leq \varepsilon C_2 \|u_k\|^m_m + C_\varepsilon \|u_k\|^{p_\varepsilon}_{p_\varepsilon },
		\end{equation*}
		together with  $W_V^{2,m}(\mathbb{R}^N)\hookrightarrow L^\theta (\mathbb{R}^N),$ $m\leq \theta  \leq  m _\ast$ replacing \eqref{eqreplace} in the argument. Moreover, since for $I'(u) = 0,$ $ u \in W_V^{2,m}(\mathbb{R}^N),$ using
		\begin{equation*}
			I(u) = \frac{1}{m}\int_{\mathbb{R}^N}f(x,u) u - m F(x,u) \dx\geq 0,
		\end{equation*}
		in \eqref{eq_replace}, we have the attainability of $\mathcal{G}_{\mathcal{S}}.$
	\end{proof}

\end{document}